\documentclass{amsart}
\usepackage{graphicx,color,pdfsync}
\usepackage{amssymb,enumerate}
\usepackage[bookmarksnumbered,colorlinks,plainpages,backref]{hyperref}
\usepackage{ulem}

\newtheorem{thm}{Theorem}[section]
\newtheorem{cor}[thm]{Corollary}
\newtheorem{lem}[thm]{Lemma}
\newtheorem{prop}[thm]{Proposition}
\theoremstyle{definition}
\newtheorem{defn}[thm]{Definition}
\theoremstyle{remark}
\newtheorem{rem}[thm]{Remark}
\numberwithin{equation}{section}
\newcommand{\norm}[1]{\left\Vert#1\right\Vert}
\newcommand{\abs}[1]{\left\vert#1\right\vert}
\newcommand{\set}[1]{\left\{#1\right\}}

\newcommand{\dbar}{\bar\partial}
\newcommand{\ddbar}{\partial\bar\partial}

\DeclareMathOperator{\dom}{Dom}

\DeclareMathOperator{\re}{Re}
\DeclareMathOperator{\im}{Im}

\DeclareMathOperator{\Span}{span}

\DeclareMathOperator{\Hess}{Hess}

\DeclareMathOperator{\Tr}{Tr}

\newcommand{\Om}{\Omega}

\DeclareMathOperator{\Rre}{Re}

\newcommand{\p}{\partial}

\newcommand{\z}{\bar z}

\newcommand{\C}{\mathbb C}

\begin{document}

\title[Maximal Estimates on Non-pseudoconvex domains]{Maximal Estimates for the $\bar\partial$-Neumann Problem on Non-pseudoconvex domains}%
\author{Phillip S. Harrington and Andrew Raich}%
\address{SCEN 309, 1 University of Arkansas, Fayetteville, AR 72701}%
\email{psharrin@uark.edu, araich@uark.edu}%
\thanks{This work was supported by a grant from the Simons Foundation (707123, ASR)}%



\begin{abstract}
It is well known that elliptic estimates fail for the $\bar\partial$-Neumann problem. Instead, the best that one can hope for is that derivatives in every direction but one can be estimated by the associated Dirichlet form, and when this happens, we say that the $\bar\partial$-Neumann problem satisfies maximal estimates. In the pseudoconvex case, 
a necessary and sufficient geometric condition for maximal estimates has been derived by Derridj (for $(0,1)$-forms) and Ben Moussa (for $(0,q)$-forms when $q\geq 1$).  
In this paper, we explore necessary conditions and sufficient conditions for maximal estimates in the non-pseudoconvex case.  We also discuss when the necessary conditions and sufficient conditions agree and provide examples. Our results subsume the earlier known results from the pseudoconvex case.
\end{abstract}
\maketitle


\section{Introduction}

The $\bar\partial$-Neumann problem is the model example of a non-coercive boundary value problem, in that it consists of a second-order operator which is elliptic in the interior, but due to the boundary condition solutions gain at most one derivative in the Sobolev scale.  See \cite{ChSh01} or \cite{Str10} for a detailed exposition of the $\bar\partial$-Neumann problem.

Given that elliptic estimates fail for the $\bar\partial$-Neumann problem, the best that one can hope for would be elliptic estimates in all directions but one.  Such estimates are known as maximal estimates (see Definition \ref{defn:maximal_estimate} below for the precise definition).  For pseudoconvex boundaries, these have been completely characterized by Derridj \cite{Derr78} for $(0,1)$-forms and by Ben Moussa \cite{Ben00} for $(0,q)$-forms when $q\geq 1$.  Although we will focus on the $\bar\partial$-Neumann problem in this paper, the analogous problem for the boundary operator $\Box_b$ has been completely characterized by Grigis and Rothschild in \cite{GrRo88}.  Maximal estimates have many applications.  In addition to those given by Derridj and Ben Moussa, see also \cite{CSS20}, \cite{Koe02}, or \cite{Liu19}.

On non-pseudoconvex domains, much less is known.  Andreotti and Grauert \cite{AnGr62} introduced the natural generalization of strict pseudoconvexity in the study of $(0,q)$-forms in manifolds of complex dimension $n$: the Levi-form must have at least $q+1$ negative or $n-q$ positive eigenvalues.  We adopt the terminology of Folland and Kohn \cite{FoKo72} in referring to this property as $Z(q)$.  It is not difficult to see that a bounded domain satisfying $Z(1)$ at every boundary point must be strictly pseudoconvex.  It follows from work of H\"ormander \cite{Hor65} that $Z(q)$ is equivalent to optimal estimates for the $\bar\partial$-Neumann operator acting on $(0,q)$-forms.  We will see in Theorem \ref{thm:Z_q} that maximal estimates on $(0,q)$-forms are equivalent to $Z(q)$ whenever the Levi-form is non-trivial.

In the degenerate case, it is natural to study $(0,q)$-forms on domains with at least $q+1$ non-positive or $n-q$ non-negative eigenvalues, but most known results seem to require a stronger condition.  See \cite{AhBaZa06}, \cite{Bri06}, \cite{EaSu80}, \cite{HaRa15}, and \cite{Ho91} for a wide range of results using an equally wide range of sufficient conditions generalizing $Z(q)$ to the degenerate case.  As far as we know, no study has yet been carried out on maximal estimates for $(0,q)$-forms on non-pseudoconvex boundaries.  Our hope in the present paper is that this study will shed some light on the appropriate conditions for the $L^2$-theory for $(0,q)$-forms on non-pseudoconvex domains.

We note that maximal estimates are local estimates.  If a domain $\Omega$ admits a maximal estimate on $(0,q)$-forms on a neighborhood $U$ of a boundary point $p$ but $\Omega$ is not pseudoconvex on $U$, there is no guarantee that we can find a domain $\tilde\Omega$ agreeing with $\Omega$ on $U$ such that the $\bar\partial$-Neumann problem is solvable on $\tilde\Omega$.  The authors have explored the technical issues involved in constructing such an extension in \cite{HaRa15}.  If, for example, the Levi-form has at least $q+1$ negative eigenvalues on $\partial\Omega\cap U$, then the domain $\tilde\Omega$ must be unbounded, which introduces additional technical issues (see \cite{HaRa17z} and \cite{HaRa18}).  Although we will not explore this issue further in the present paper, we note that our strongest results (e.g., Theorems \ref{thm:Z_q} and \ref{thm:almost_pseudoconvex}), have hypotheses that are strong enough to imply the existence of such an extension when the Levi-form has at least $n-q$ non-negative eigenvalues on $\partial\Omega\cap U$.

In Section \ref{sec:necessary_condition}, we will adapt an argument of H\"ormander \cite{Hor65} to derive a necessary condition for maximal estimates on $(0,q)$-forms that generalizes the conditions derived by Derridj and Ben Moussa.  If we denote the eigenvalues of the Levi-form, arranged in non-decreasing order, by $\{\lambda_1,\ldots,\lambda_{n-1}\}$, then on pseudoconvex domains maximal estimates on $(0,q)$-forms near some point $p\in\partial\Omega$ are equivalent to the existence of some constant $\epsilon>0$ such that $\sum_{j=1}^q\lambda_j\geq\epsilon\sum_{j=1}^{n-1}\lambda_j$ on some neighborhood of $p$ in $\partial\Omega$ \cite{Ben00}.  If we replace pseudoconvexity with the requirement that the Levi-form have at least $n-q$ non-negative eigenvalues in a neighborhood of $p$, then we will show in Theorem \ref{thm:necessary_condition} that maximal estimates on $(0,q)$-forms near $p$ imply the existence of $\epsilon>0$ such that $\sum_{\{1\leq j\leq q:\lambda_j\geq 0\}}\lambda_j\geq\epsilon\sum_{j=1}^{n-1} |\lambda_j|$ on some neighborhood of $p$ in $\partial\Omega$.  Unlike the conditions of Derridj and Ben Moussa, it is not yet clear whether this necessary condition is also sufficient, although Theorems \ref{thm:Z_q} and \ref{thm:almost_pseudoconvex} provide a wide range of special cases in which our necessary condition is also sufficient.

Our sufficient conditions require additional discussion.  For most such results, the starting point is the Morrey-Kohn identity.  Suppose that $\Omega\subset\mathbb{C}^n$ is a domain with $C^2$ boundary and $\rho$ is a $C^2$ defining function for $\Omega$.  If $\bar\partial$ denotes the maximal $L^2$ extension of the Cauchy-Riemann operator, $\bar\partial^*$ denotes the Hilbert space adjoint, $1\leq q\leq n$, $\mathcal{I}_q$ denotes the set of increasing multi-indices of length $q$, and $u\in C^1_{0,q}(\overline\Omega)\cap\dom\dbar^*$, then we have the Morrey-Kohn identity (a special case of the Morrey-Kohn-H\"ormander identity given by Proposition 4.3.1 in \cite{ChSh01} or Proposition 2.4 in \cite{Str10}):
\begin{multline}
\label{eq:Morrey_Kohn}
  \norm{\dbar u}^2_{L^2(\Omega)}+\norm{\dbar^* u}^2_{L^2(\Omega)}=\\
  \sum_{j=1}^n\sum_{J\in\mathcal{I}_q}\norm{\frac{\partial}{\partial\bar z_j}u_J}^2_{L^2(\Omega)}+\sum_{j,k=1}^n\sum_{I\in\mathcal{I}_{q-1}}\int_{\partial\Omega}|\nabla\rho|^{-1}u_{jI}\rho_{j\bar k}\overline{u_{kI}} \, d\sigma.
\end{multline}
In both the study of maximal estimates for pseudoconvex domains and the study of $L^2$ theory for $\bar\partial$ on non-pseudoconvex domains, it is useful to integrate by parts in the first term on the right-hand side.  As shown in \cite{HaRa15}, it is useful to prescribe this integration by parts via a Hermitian matrix of functions $\Upsilon$ with the property that every eigenvalue of $\Upsilon$ is bounded between $0$ and $1$.  Then we can decompose the gradient into two non-negative terms:
\begin{multline*}
  \sum_{j=1}^n\sum_{J\in\mathcal{I}_q}\norm{\frac{\partial}{\partial\bar z_j}u_J}^2_{L^2(\Omega)}=\sum_{j,k=1}^n\sum_{J\in\mathcal{I}_q}\int_\Omega(\delta_{jk}-\Upsilon^{\bar k j})\frac{\partial}{\partial\bar z_k}u_J \frac{\partial}{\partial z_j}\overline{u_J} \, dV\\
  +\sum_{j,k=1}^n\sum_{J\in\mathcal{I}_q}\int_\Omega\Upsilon^{\bar k j}\frac{\partial}{\partial\bar z_k}u_J \frac{\partial}{\partial z_j}\overline{u_J} \, dV.
\end{multline*}
Postponing the technical details for Section \ref{sec:Basic_Identity}, we observe that integrating by parts twice in the second term gives us
\begin{multline*}
  \sum_{j=1}^n\sum_{J\in\mathcal{I}_q}\norm{\frac{\partial}{\partial\bar z_j}u_J}^2_{L^2(\Omega)}=\sum_{j,k=1}^n\sum_{J\in\mathcal{I}_q}\int_\Omega(\delta_{jk}-\Upsilon^{\bar k j})\frac{\partial}{\partial\bar z_k}u_J \frac{\partial}{\partial z_j}\overline{u_J} \, dV\\
  +\sum_{j,k=1}^n\sum_{J\in\mathcal{I}_q}\int_\Omega\Upsilon^{\bar k j}\frac{\partial}{\partial z_j}u_J \frac{\partial}{\partial\bar z_k}\overline{u_J} \, dV+\cdots,
\end{multline*}
where we have omitted the lower order terms.  To avoid boundary terms, we must choose $\Upsilon$ so that we are only integrating by parts with respect to tangential derivatives (see \eqref{eq:Upsilon_mixed_vanishes} below).  This implies that on the boundary $\Upsilon$ must have at least one eigenvalue equal to zero, but if the remaining eigenvalues are uniformly bounded away from zero and one, then we have an estimate for every derivative of $u$ orthogonal to $Z=\sum_{j=1}^n\frac{\partial\rho}{\partial\bar z_j}\frac{\partial}{\partial z_j}$.  This is precisely what we need to obtain maximal estimates.  However, two integrations by parts will introduce two derivatives of the coefficients of $\Upsilon$, which is why in \cite{HaRa15} we assume that the coefficients of $\Upsilon$ are $C^2$.  Indeed, in the proof of Theorem \ref{thm:Z_q}, we will use this technique to prove maximal estimates on $(0,q)$-forms on any $C^3$ boundary satisfying $Z(q)$.  However, in Lemma \ref{lem:gradient_transformation} below, we will see that it suffices for $\Upsilon$ to have weak derivatives in certain Sobolev spaces and for certain functions of these weak derivatives to be bounded (see also \cite{ChHa18}, where a related observation was used to study closed range of $\bar\partial$ near pseudoconcave boundaries with low boundary regularity).  This may seem like a needless complication, but in Proposition \ref{prop:ex_2}, we will construct an explicit example where maximal estimates hold but it is impossible to prove that maximal estimates hold using the above method when $\Upsilon$ has $C^2$ coefficients.  Furthermore, this will allow us to relax the boundary regularity in certain special cases and show that, for example, the results of Derridj and Ben Moussa hold on domains with $C^{2,\alpha}$ boundaries for any $0<\alpha<1$.

The implications of Proposition \ref{prop:ex_2} for the $L^2$ theory of the $\bar\partial$-Neumann problem on non-pseudoconvex domains are significant.  As the authors have shown in \cite{HaRa20}, if we require $\Upsilon$ to have $C^2$ coefficients, then we are placing significant restrictions on the boundary geometry near points where the Levi-form degenerates.  In the present paper, we have provided evidence that these geometric restrictions are unlikely to be necessary, but the methods of proof will require greater subtlety than has been used in the past.  Indeed, constructing the $\Upsilon$ used in the proof of Proposition \ref{prop:ex_2} is not trivial (especially compared to the relatively straight-forward construction used in Theorem \ref{thm:Z_q}).  However, a deeper understanding of these methods may make it possible to derive more natural sufficient conditions for maximal estimates, or even show that the necessary condition given in Theorem \ref{thm:necessary_condition} is also sufficient.

\section{Definitions and Key Results}

For $1\leq q\leq n$, let $\mathcal{I}_q$ denote the set of all increasing multi-indices over $\{1,\ldots,n\}$ of length $q$.  For a $(0,q)$-form $u$, we write $u=\sum_{J\in\mathcal{I}_q} u_J \, d\bar z^J$.  We may extend the definition of $u_J$ to coefficients with non-increasing multi-indices by requiring $u_J$ to be skew-symmetric with respect to the indices in $J$.  When each $u_J$ is $C^1$, we may write
\[
  \bar\partial u=\sum_{J\in\mathcal{I}_q}\sum_{j=1}^n\frac{\partial}{\partial\bar z_j}u_J\, d\bar z_j\wedge d\bar z^J,
\]
with the formal $L^2$ adjoint
\[
  \vartheta u=-\sum_{I\in\mathcal{I}_{q-1}}\sum_{j=1}^n\frac{\partial}{\partial z_j}u_{jI}\, d\bar z^I.
\]

We normalize our Euclidean metric for $\mathbb{C}^n$ so that
\[
  \left<dz_j,dz_k\right>=\delta_{jk}\text{ and }\left<\frac{\partial}{\partial z_j},\frac{\partial}{\partial z_k}\right>=\delta_{jk}\text{ for all }1\leq j,k\leq n.
\]
Let $\, dV$ denote the corresponding volume element, and for $\Omega\subset\mathbb{C}^n$ with $C^1$ boundary and $u,v\in L^2_{0,q}(\Omega)$, we write
\[
  \left(u,v\right)_{L^2(\Omega)}=\sum_{J\in\mathcal{I}_q}\int_\Omega u_J\overline{v_J}\, dV
\]
and
\[
  \norm{u}^2_{L^2(\Omega)}=\sum_{J\in\mathcal{I}_q}\int_\Omega|u_J|^2 \, dV.
\]
We use $\dom\dbar\subset L^2_{0,q}(\Omega)$ to denote the set of all $u\in L^2_{0,q}(\Omega)$ such that $\dbar u$ exists in the distribution sense and $\dbar u\in L^2_{0,q+1}(\Omega)$.  We let $\dbar^*$ denote the Hilbert space adjoint with domain $\dom\dbar^*$.  As is well-known (see, e.g., Lemma 4.2.1 in \cite{ChSh01} or (2.9) in \cite{Str10}), $u\in C^1_{0,q}(\overline\Omega)$ is in $\dom\dbar^*$ if and only if
\[
  \sum_{j=1}^n\frac{\partial\rho}{\partial z_j}u_{jI}=0\text{ on }\partial\Omega\text{ for all }I\in\mathcal{I}_{q-1},
\]
and in this case $\dbar^* u=\vartheta u$.

We let $\, d\sigma$ denote the surface measure on $\partial\Omega$ induced from the ambient Euclidean metric.

We may now give a precise definition of a maximal estimate.
\begin{defn}
\label{defn:maximal_estimate}
  Let $\Omega\subset\mathbb{C}^n$ be a domain with $C^{1,1}$ boundary and let $1\leq q\leq n$.  We say that $\Omega$ admits a
\emph{maximal estimate on $(0,q)$-forms} near $p\in\partial\Omega$ if $p$ admits a neighborhood $U$ and constants $A,B>0$ such that if $\{L_1,\ldots,L_n\}$ is an orthonormal basis for $T^{1,0}(U)$ with Lipschitz coefficients satisfying $L_j|_{\partial\Omega}\in T^{1,0}(\partial\Omega\cap U)$ for every $1\leq j\leq n-1$, then
  \begin{multline}
  \label{eq:maximal_estimate}
    \sum_{j=1}^n\sum_{J\in\mathcal{I}_q}\norm{\frac{\partial}{\partial\bar z_j}u_J}^2_{L^2(\Omega)}+\sum_{j=1}^{n-1}\sum_{J\in\mathcal{I}_q}\norm{L_j u_J}^2_{L^2(\Omega)}\leq\\
     A\left(\norm{\dbar u}^2_{L^2(\Omega)}+\norm{\dbar^* u}^2_{L^2(\Omega)}\right)+B\norm{u}^2_{L^2(\Omega)}
  \end{multline}
  for every $u\in C^1_{0,q}(\overline\Omega)\cap\dom\dbar^*$ supported in $\overline\Omega\cap U$.
\end{defn}
We will see that the constant $A>0$ in \eqref{eq:maximal_estimate} plays a significant role in both our necessary and sufficient conditions.  We will see that this definition and the optimal value of $A$ are both independent of the choice of orthonormal basis in Lemma \ref{lem:gradient_comparison}.  Note that when $q=n$, the $\bar\partial$-Neumann problem is elliptic, so $\Omega$ always admits a maximal estimate on $(0,n)$-forms near any boundary point.

Now that we have a precise definition for maximal estimates, we are able to state our necessary condition:
\begin{thm}
\label{thm:necessary_condition}
  Let $\Omega\subset\mathbb{C}^n$ be a domain with $C^{2,\alpha}$ boundary for some $\alpha>0$ and let $1\leq q\leq n-1$.  Assume that $\Omega$ admits maximal estimates on $(0,q)$-forms near $p\in\partial\Omega$.  Let $U$ and $A$ be as in Definition \ref{defn:maximal_estimate}, and let the eigenvalues of the Levi-form on $U\cap\partial\Omega$ with respect to a $C^{2,\alpha}$ defining function $\rho$ be given by $\{\lambda_1,\ldots,\lambda_{n-1}\}$, arranged in non-decreasing order.  Then
  \begin{equation}
  \label{eq:necessary_condition}
    \sum_{j=1}^q\lambda_j+\sum_{\{1\leq j\leq n-1:\lambda_j<0\}}|\lambda_j|\geq\frac{1}{A}\sum_{j=1}^{n-1} |\lambda_j|
  \end{equation}
  on $U\cap\partial\Omega$.
\end{thm}

In order to state our sufficient condition, we need additional structures.  For an open set $U\subset\mathbb{R}^n$, an integer $k\geq 0$, and a real number $1\leq p\leq\infty$, we define the Sobolev space $W^{k,p}(U)$ to be the set of all $u\in L^p(U)$ such that all derivatives of $u$ of order $k$ or less exist in the weak sense and belong to $L^p(U)$.  For $u\in W^{1,p}(U)$, we let $\nabla u$ denote the gradient of $u$.  For $u\in W^{2,p}(U)$, we let $\nabla^2 u$ denote the Hessian of $u$.  For $Z,W\in T^{1,0}(\mathbb{C}^n)$ and $f\in C^2(\mathbb{C}^n)$, if we write $Z=\sum_{j=1}^n a^j(z)\frac{\partial}{\partial z_j}$ and $W=\sum_{j=1}^n b^j(z)\frac{\partial}{\partial z_j}$, then we may define the \emph{complex Hessian}
\begin{equation}
\label{eq:hessian_alternate}
  \Hess(Z,\bar W)f(z)=\sum_{j,k=1}^n a^j(z)\frac{\partial^2 f}{\partial z_j\partial\bar z_k}(z)\bar b^k(z).
\end{equation}

\begin{defn}
\label{eq:Upsilon_defn}
  For $\Omega\subset\mathbb{C}^n$ with $C^{1,1}$ boundary and $U\subset\mathbb{C}^n$ satisfying $U\cap\partial\Omega\neq\emptyset$, let $\mathcal{M}^1_\Omega(U)$ denote the space of Hermitian $n\times n$ matrices $\Upsilon$ such that
  \begin{enumerate}
    \item each entry of $\Upsilon$ is an element of $L^\infty(U)\cap W^{1,1}(U)$,
    \item $\Upsilon$ and $I-\Upsilon$ are positive semi-definite almost everywhere on $U$,
    \item for $1\leq j\leq n$, $\Upsilon^j\in L^\infty(U)$, where $\Upsilon^j$ is defined by
    \begin{equation}
    \label{eq:Upsilon_j_defined}
      \Upsilon^j=\sum_{k=1}^n\frac{\partial}{\partial\bar z_k}\Upsilon^{\bar k j}
    \end{equation}
    almost everywhere in $U$, and
    \item for any $C^{1,1}$ defining function $\rho$ for $\Omega$ on $U$,
    \begin{equation}
    \label{eq:Upsilon_normal_vanishes}
      \abs{\sum_{j,k=1}^n\frac{\partial\rho}{\partial\bar z_k}\Upsilon^{\bar k j}\frac{\partial\rho}{\partial z_j}}\leq O(|\rho|^2)
    \end{equation}
    almost everywhere on $U$.
  \end{enumerate}
  For $0\leq\eta\leq 1$, let $\mathcal{M}^2_{\Omega,\eta}(U)$ denote the set of $\Upsilon\in\mathcal{M}^1_\Omega(U)$ such that
  \begin{enumerate}
    \item each entry of $\Upsilon$ is an element of $L^\infty(U)\cap W^{1,2}(U)\cap W^{2,1}(U)$ and
    \item $\Theta_{\Upsilon,\eta}\in L^\infty(U)$, where $\Theta_{\Upsilon,\eta}$ is defined by
      \begin{multline}
      \label{eq:Theta_defined}
        \Theta_{\Upsilon,\eta}=\sum_{j,k=1}^n\frac{\partial^2}{\partial z_j\partial\bar z_k}\left(\Upsilon^{\bar k j}-\eta\sum_{\ell=1}^n\Upsilon^{\bar k\ell}\Upsilon^{\bar\ell j}\right)\\
        +\eta\sum_{j,k,\ell=1}^n\left(\frac{\partial}{\partial z_j}\Upsilon^{\bar k\ell}\right)\left(\frac{\partial}{\partial\bar z_k}\Upsilon^{\bar\ell j}\right)
        -(1-\eta)\sum_{j=1}^n\abs{\Upsilon^j}^2.
      \end{multline}
  \end{enumerate}
\end{defn}
Since $\Upsilon\in\mathcal{M}^1_\Omega(U)$ is positive semi-definite and each element of $\Upsilon$ is in $L^\infty(U)$, \eqref{eq:Upsilon_normal_vanishes} implies that for any $C^{1,1}$ defining function $\rho$ for $\Omega$ on $U$,
\begin{equation}
\label{eq:Upsilon_mixed_vanishes}
  \abs{\sum_{k=1}^n\frac{\partial\rho}{\partial\bar z_k}\Upsilon^{\bar k j}}\leq O(|\rho|)\text{ for all }1\leq j\leq n
\end{equation}
almost everywhere on $U$.  Since each element of $\Upsilon$ is in $L^\infty(U)$, it is easy to check that \eqref{eq:Upsilon_normal_vanishes} and \eqref{eq:Upsilon_mixed_vanishes} are together independent of the choice of $C^{1,1}$ defining function $\rho$.

In the special case in which $\Upsilon$ is a projection matrix, the component $\Upsilon^{\bar k j}=\sum_{\ell=1}^n\Upsilon^{\bar k\ell}\Upsilon^{\bar\ell j}$ for all $1\leq j,k\leq n$, so the first term in \eqref{eq:Theta_defined} vanishes when $\eta= 1$.  In this case, it suffices to assume that $\Upsilon$ has Lipschitz coefficients.  This is precisely the case explored in \cite{ChHa18}.

Now we may state our sufficient condition:
\begin{prop}
\label{prop:sufficient_condition}
  Let $\Omega\subset\mathbb{C}^n$ be a domain with $C^{1,1}$ boundary and let $1\leq q\leq n-1$. Let $\rho$ be a $C^{1,1}$ defining function for $\Omega$.  Assume that for some $p\in\partial\Omega$ there exists a neighborhood $U$ of $p$, constants $A>2$ and $0\leq\eta\leq 1$, and $\Upsilon\in\mathcal{M}^2_{\Omega,\eta}(U)$ such that:
  \begin{enumerate}
    \item counting multiplicity, $\Upsilon$ has $n-1$ eigenvalues in the interval $(1/A,1-1/A)$ almost everywhere on $U$, and
    \item if $\{\lambda_1,\ldots,\lambda_{n-1}\}$ denote the eigenvalues of the Levi-form (computed with $\rho$) arranged in non-decreasing order, then
        \begin{equation}
        \label{eq:weak_z_q}
          \sum_{j=1}^q\lambda_j-\sum_{j,k=1}^n\Upsilon^{\bar k j}\rho_{j\bar k}\geq 0
        \end{equation}
        almost everywhere on $U\cap\partial\Omega$.
  \end{enumerate}
  Then $\Omega$ admits a maximal estimate on $(0,q)$-forms near $p$.
\end{prop}

\begin{rem} The crucial new hypothesis in Proposition \ref{prop:sufficient_condition} is (1). Condition (2) is one of the hypotheses
of the weak $Z(q)$ condition developed by the authors as a sufficient condition for closed range of $\dbar$ on $(0,q)$-forms
\cite{HaRa15,HaRa19}.  In those papers, it sufficed for the eigenvalues of $\Upsilon$ to lie in the interval $[0,1]$.
\end{rem}

It is not clear whether the necessary condition given by Theorem \ref{thm:necessary_condition} and the sufficient condition given by Proposition \ref{prop:sufficient_condition} are equivalent, but we do have some important special cases.  When the Levi-form is non-trivial, we may use these results to show:
\begin{thm}
\label{thm:Z_q}
  Let $\Omega\subset\mathbb{C}^n$ be a domain with $C^3$ boundary and let $1\leq q\leq n-1$.  Assume that for some $p\in\partial\Omega$ the Levi-form for $\partial\Omega$ has at least one non-vanishing eigenvalue at $p$.  Then $\Omega$ admits a maximal estimate on $(0,q)$-forms near $p$ if and only if $\Omega$ satisfies $Z(q)$ at $p$.
\end{thm}
Alternatively, Theorem \ref{thm:Z_q} could be written to state that when at least one eigenvalue is non-vanishing, \eqref{eq:necessary_condition} is both necessary and sufficient for maximal estimates on $(0,q)$-forms.  The following result gives another class of domains for which \eqref{eq:necessary_condition} is both a necessary and sufficient condition for maximal estimates on $(0,q)$-forms.
\begin{thm}
\label{thm:almost_pseudoconvex}
  Let $\Omega\subset\mathbb{C}^n$ be a domain with $C^{2,\alpha}$ boundary for some $0<\alpha<1$, and let $1\leq q\leq n-1$.  Let $\rho$ be a $C^{2,\alpha}$ defining function for $\Omega$, and let $\{\lambda_j\}_{j=1}^{n-1}$ be the eigenvalues of the Levi-form of $\partial\Omega$ with respect to $\rho$ arranged in non-decreasing order.  Assume that for some $p\in\partial\Omega$, there exists a neighborhood $U$ of $p$ and a constant $\epsilon>0$ such that either
  \begin{enumerate}
    \item
    \begin{equation}
    \label{eq:almost_pseudoconvex}
      \sum_{j=1}^q\lambda_j\geq\epsilon\sum_{\{1\leq j\leq n-1:\lambda_j<0\}}|\lambda_j|\text{ on }\partial\Omega\cap U
    \end{equation}
    or
    \item
    \begin{equation}
    \label{eq:almost_pseudoconcave}
      -\sum_{j=q+1}^{n-1}\lambda_j\geq\epsilon\sum_{\{1\leq j\leq n-1:\lambda_j>0\}}\lambda_j\text{ on }\partial\Omega\cap U.
    \end{equation}
  \end{enumerate}
  Then $\Omega$ admits a maximal estimate on $(0,q)$-forms near $p$ if and only if \eqref{eq:necessary_condition} holds on some neighborhood of $p$.
\end{thm}
Observe that \eqref{eq:almost_pseudoconvex} holds on all pseudoconvex domains, so this generalizes the results of Derridj \cite{Derr78} for $(0,1)$-forms and Ben Moussa \cite{Ben00} for $(0,q)$-forms.  Conversely, any domain satisfying \eqref{eq:almost_pseudoconvex} must be weakly $q$-convex in the sense of Ho \cite{Ho91}.  Note that \eqref{eq:almost_pseudoconcave} holds for all pseudoconcave domains.

\section{A Special Defining Function}

To reduce the required boundary regularity of our key results, we will need a special defining function:
\begin{lem}
\label{lem:rho_derivative_estimate}
  Let $\Omega\subset\mathbb{R}^n$ be a domain with $C^{m,\alpha}$ boundary, $m\geq 1$ and $0\leq\alpha\leq 1$, and let $U\subset\mathbb{R}^n$ satisfy $U\cap\partial\Omega\neq\emptyset$.  Then there exists a defining function $\rho$ for $\Omega$ on $U$ such that $\rho\in C^{m,\alpha}(U)\cap C^\infty(U\backslash\partial\Omega)$ and for any differential operator $D^k$ of order $k\geq m$, we have
  \begin{equation}
  \label{eq:rho_derivative_estimate}
    |D^k\rho(x)|\leq O(|\rho(x)|^{m+\alpha-k})\text{ for all }x\in U\backslash\partial\Omega.
  \end{equation}
\end{lem}

\begin{proof}
  After a translation and rotation, we may assume that $0\in\overline U\cap\partial\Omega$ and $T_0(\partial\Omega)=\Span\set{\frac{\partial}{\partial x_j}}_{j=1}^{n-1}$.  Choose a neighborhood $V$ of $0$ that is sufficiently small so that
  \[
    \Omega\cap V=\{x\in V:\varphi(x')<x_n\},
  \]
  where $x'=(x_1,\ldots,x_{n-1})$ and $\varphi\in C^{m,\alpha}(\mathbb{R}^{n-1})$.  In our chosen coordinates, $\varphi(0)=0$ and $\nabla\varphi(0)=0$.

  Let $\chi$ be a smooth, radially symmetric function supported in the unit ball in $\mathbb{R}^{n-1}$ satisfying $\int_{\mathbb{R}^{n-1}}\chi=1$.  For $t\in\mathbb{R}\backslash\{0\}$ and $x'\in\mathbb{R}^{n-1}$, set $\chi_t(x')=|t|^{1-n}\chi(t^{-1}x')$, so that $\int_{\mathbb{R}^{n-1}}\chi_t=1$.  For all $x\in V$ and $t\in\mathbb{R}$, we define
  \[
    f(x,t)=\begin{cases}(\varphi\ast\chi_{t})(x')-x_n-t&t\neq 0\\\varphi(x')-x_n&t=0\end{cases}.
  \]
  Using a standard change of coordinates for convolutions, we have the equivalent statement
  \[
    f(x,t)=\int_{\mathbb{R}^{n-1}}\varphi(x'-ty')\chi(y')\, dy'-x_n-t
  \]
  for all $x\in V$ and $t\in\mathbb{R}$.  Differentiating this in $t$ gives us
  \[
    \frac{\partial f}{\partial t}(x,t)=-\int_{\mathbb{R}^{n-1}}y'\cdot\nabla\varphi(x'-ty')\chi(y')\, dy'-1.
  \]
  Since $\chi$ is radially symmetric, $\int_{\mathbb{R}^{n-1}}y_j\chi(y')\, dy'=0$ for all $1\leq j\leq n-1$, and hence $\int_{\mathbb{R}^{n-1}}y'\cdot\nabla\varphi(x')\chi(y')\, dy'=0$ for all $x\in V$.  Since $\nabla\varphi$ is uniformly continuous on any compact subset of $\mathbb{R}^{n-1}$, there exists $T>0$ such that
  \[
    \abs{\int_{\mathbb{R}^{n-1}}y'\cdot\nabla\varphi(x'-ty')\chi(y')\, dy'}\leq\frac{1}{2}\text{ whenever }0\leq t\leq T\text{ and }x\in V.
  \]
  Hence, $\frac{\partial f}{\partial t}(x,t)\leq-\frac{1}{2}<0$ whenever $0\leq t\leq T$ and $x\in V$.  By the Implicit Function Theorem, there exists a neighborhood $W$ of the origin and a unique function $\rho:W\rightarrow\mathbb{R}$ such that $f(x,\rho(x))=0$ and $\nabla\rho(x)=-\left(\frac{\partial f}{\partial t}(x,\rho(x))\right)^{-1}\nabla_x f(x,\rho(x))$ for all $x\in W$.  Since $\rho$ is unique, we must have $\rho(x)=0$ whenever $x\in\partial\Omega\cap W$, and hence $\nabla\rho(x)=\nabla(\varphi(x')-x_n)$ whenever $x\in\partial\Omega\cap W$.  This means that $\rho$ is a $C^1$ defining function for $\Omega$ on $W$.  Since $f\in C^m(\mathbb{R}^{n}\times\mathbb{R})\cap C^\infty(\mathbb{R}^n\times(\mathbb{R}\backslash\{0\}))$, it is not difficult to check that $\rho\in C^m(W)\cap C^\infty(W\backslash\partial\Omega)$.

  For a multi-index $J$ over $n-1$ variables, let $\chi^J(y')=y^J\chi(y')$, and set $\chi_t^J(y')=|t|^{1-n}t^{-|J|}y^J\chi(t^{-1}y')$.  Let $D_{n-1}^m$ be a differential operator of order $m$ in $n-1$ variables.  Any order $m$ derivative of $f(x,t)$ in $n+1$ variables will be a linear combination of terms of the form $(D_{n-1}^m\varphi\ast\chi^J_t)$ (with an additional constant term if $m=1$), where $|J|\leq m$.  For $k\geq m+1$, if $D_{n+1}^{k-m}$ is a differential operator of order $k$ in $n+1$ variables, then we may write any order $k$ derivative of $f(x,t)$ as a linear combination of terms of the form $(D_{n-1}^m\varphi\ast D_{n+1}^{k-m}\chi^J_t)$.  Observe that $\int_{\mathbb{R}^{n-1}}\chi^J_t(x'-y')\, dy'$ is independent of $x$, so $1\ast D_{n+1}^{k-m}\chi^J_t\equiv 0$ when $k>m$.  Hence, we have
  \[
    (D_{n-1}^m\varphi\ast D_{n+1}^{k-m}\chi^J_t)(x')=((D_{n-1}^m\varphi(\cdot)-D_{n-1}^m\varphi(x'))\ast D_{n+1}^{k-m}\chi^J_t(\cdot))(x').
  \]
  Using the H\"older continuity of $m$th derivatives of $\varphi$, we have
  \[
    \abs{(D_{n-1}^m\varphi\ast D_{n+1}^{k-m}\chi^J_t)(x')}\leq|t|^\alpha\norm{\varphi}_{C^{m,\alpha}(\mathbb{R}^{n-1})}\norm{D_{n+1}^{k-m}\chi^J_t}_{L^1(\mathbb{R}^{n-1})}.
  \]
  It is not difficult to check that $\norm{D_{n+1}^{k-m}\chi^J_t}_{L^1(\mathbb{R}^{n-1})}\leq O(|t|^{m-k})$, and hence we have
  \[
    \abs{D_{n+1}^k f(x,t)}\leq O(|t|^{m+\alpha-k})
  \]
  whenever $D_{n+1}^k$ is a differential operator of order $k$ in $n$ variables.  From here, \eqref{eq:rho_derivative_estimate} follows from the definition of $f$ and repeated applications of the chain rule.

  We may extend the domain of $\rho$ to a neighborhood of $\overline U$ via a partition of unity and Whitney's Extension Theorem.

\end{proof}

When we apply our results in $\mathbb{C}^n$, we will need a basis for tangent vectors with coefficients in the appropriate Sobolev spaces.  Simply constructing this basis on the boundary and extending to the interior is problematic, because the space of traces of elements of $W^{2,1}(\Omega)$ is not easy to characterize (unlike $W^{1,1}(\Omega)$, as shown by \cite{Gag57}).  The characterization in terms of Besov spaces given by Theorem 7.39 in \cite{AdFo03}, for example, fails when $p=1$, and even the one-sided extension operator given by integrating against the Poisson kernel will be deficient when $p=1$.  Furthermore, when characterizing regularity of boundaries, the most natural space is the class of functions with H\"older continuous derivatives at some order.  Due to these constraints, the following result is as sharp as can be expected.
\begin{lem}
\label{lem:orthonormal_coordinates}
  Let $\Omega\subset\mathbb{C}^n$ be a domain with $C^{2,\alpha}$ boundary for some $\alpha>0$ and let $p\in\partial\Omega$.  Then there exists a neighborhood $U$ of $p$ and an orthonormal basis $\{L_j\}_{j=1}^n$ for $T^{1,0}(U)$ such that $\{L_j|_{\partial\Omega\cap U}\}_{j=1}^{n-1}$ is an orthonormal basis for $T^{1,0}(\partial\Omega\cap U)$ and each $L_j$ has coefficients in $C^1(U)\cap W^{2,1}(U)$.
\end{lem}

\begin{proof}
  After a translation and rotation, we may assume that $p=0$ and $T^{1,0}_0(\partial\Omega)=\Span\set{\frac{\partial}{\partial z_j}}_{j=1}^{n-1}$.  Choose a neighborhood $U$ of $p$ that is sufficiently small so that
  \[
    \Omega\cap U=\{z\in U:\varphi(z',\re z_n)<\im z_n\},
  \]
  where $z'=(z_1,\ldots,z_{n-1})$ and $\varphi\in C^{2,\alpha}(\mathbb{C}^{n-1}\times\mathbb{R})$.  In our chosen coordinates, $\varphi(0)=0$ and $\nabla\varphi(0)=0$.  Let $\rho$ be the $C^{2,\alpha}$ defining function for $\Omega$ on $U$ given by Lemma \ref{lem:rho_derivative_estimate}.  We may assume that $U$ is sufficiently small so that $\frac{\partial\rho}{\partial z_n}\neq 0$ on $\overline{U}$.

  For $1\leq j\leq n$ and $z\in U$, set $a^j(z)=|\partial\rho(z)|^{-1}\frac{\partial\rho}{\partial\bar z_j}(z)$.  On $U$, we have $|a|=1$ and $\sum_{j=1}^n a^j\frac{\partial\rho}{\partial z_j}=|\partial\rho|$.  For $z\in U$, set
  \[
    b^j(z)=\begin{cases}\left(|\partial\rho(z)|+\abs{\frac{\partial\rho}{\partial z_n}(z)}\right)^{-1}\frac{\partial\rho}{\partial\bar z_j}(z)&1\leq j\leq n-1\\\abs{\frac{\partial\rho}{\partial z_n}(z)}^{-1}\frac{\partial\rho}{\partial\bar z_n}(z)&j=n\end{cases}.
  \]
  On $U$, we have $|b|^2=2|\partial\rho|\left(|\partial\rho|+\abs{\frac{\partial\rho}{\partial z_n}}\right)^{-1}$ and $\sum_{j=1}^n b^j\frac{\partial\rho}{\partial z_j}=|\partial\rho|$.  For $z\in U$, if we set
  \[
    L_j=\begin{cases}\frac{\partial}{\partial z_j}-\overline{a^j(z)}\sum_{k=1}^n b^k(z)\frac{\partial}{\partial z_k}&1\leq j\leq n-1\\\sum_{j=1}^n a^j(z)\frac{\partial}{\partial z_j}&j=n\end{cases},
  \]
  then we may check that $\{L_j\}_{j=1}^{n}$ is an orthonormal basis for $T^{1,0}(U)$ with coefficients in $C^{1,\alpha}(U)$ and $L_j\rho\equiv 0$ on $U$ for all $1\leq j\leq n-1$.  Since $|\rho|^{\alpha-1}$ is an integrable function on $U$, \eqref{eq:rho_derivative_estimate} guarantees that the coefficients of each $L_j$ must lie in $W^{2,1}(U)$.
\end{proof}

\section{The Basic Identity}
\label{sec:Basic_Identity}

Our goal for this section is to use integration by parts to transform the gradient term in \eqref{eq:Morrey_Kohn} into the special gradient defined by \eqref{eq:G_Upsilon}.

\begin{lem}
\label{lem:gradient_transformation}
  Let $\Omega\subset\mathbb{C}^n$ be a domain with $C^{1,1}$ boundary.  Let $p\in\partial\Omega$ and let $U$ be a neighborhood of $p$.  Let $\rho$ be a $C^{1,1}$ defining function for $\Omega$.  Let $0\leq\eta\leq 1$, and let $\Upsilon\in\mathcal{M}^2_{\Omega,\eta}$.  For $0\leq q\leq n$ and $u,v\in W^{1,2}_{0,q}(\Omega)$ supported in $\overline\Omega\cap U$, let
  \begin{multline}
  \label{eq:G_Upsilon}
    G_{\Upsilon,\eta}(u,v)=
    \sum_{j,k=1}^n\sum_{J\in\mathcal{I}_q}\int_{\Omega\cap U} \bar Z_{k,\eta} u_J(\delta_{j k}-\Upsilon^{\bar k j})\overline{\bar Z_{j,\eta} v_J}\, dV\\
    +\sum_{j,k=1}^n\sum_{J\in\mathcal{I}_q}\int_{\Omega\cap U} \bar Z_{j,\eta}^*u_J\Upsilon^{\bar k j}\overline{\bar Z_{k,\eta}^* v_J}\, dV,
  \end{multline}
  where
  \[
    \bar Z_{j,\eta}=\frac{\partial}{\partial\bar z_j}-\eta\Upsilon^j\text{ with adjoint }\bar Z^*_{j,\eta}=-\frac{\partial}{\partial z_j}-\eta\overline{\Upsilon^{j}}
  \]
  almost everywhere on $U$ for all $1\leq j\leq n$.  Then
  \begin{multline}
  \label{eq:gradient_transformation}
    G_{\Upsilon,\eta}(u,u)=\sum_{j=1}^n\sum_{J\in\mathcal{I}_q}\norm{\bar Z_{j,\eta-1} u_J}^2_{L^2(\Omega)}+\int_\Omega \Theta_{\Upsilon,\eta}|u|^2\, dV\\
    +\sum_{j,k=1}^n\int_{\partial\Omega} |\nabla\rho|^{-1}\frac{\partial^2\rho}{\partial z_j\partial\bar z_k}\Upsilon^{\bar k j}|u|^2\, d\sigma
  \end{multline}
  for all $u\in W^{1,2}_{0,q}(\Omega)$ supported in $\overline\Omega\cap U$.
\end{lem}

\begin{proof}
  We observe that this Lemma doesn't make use of the exterior algebra structure $\Lambda_{0,q}(\overline\Omega)$ in any meaningful way, so it will suffice to prove \eqref{eq:gradient_transformation} in the $q=0$ case and apply it to each coefficient of a $(0,q)$-form.
  We first assume that $u\in C^2(\overline\Omega)$ and $\Upsilon$ is a Hermitian $n\times n$ matrix with entries in $C^2(U)$.

  We define the crucial error terms
  \begin{align*}
    E_{\Upsilon,\eta}(u)&=\sum_{j,k=1}^n\int_{\partial\Omega} |\nabla\rho|^{-1}\frac{\partial\rho}{\partial\bar z_k}\Upsilon^{\bar k j}\left(\bar Z_{j,\eta}^*u\right)\overline{u}\, d\sigma,\\
    F_{\Upsilon,\eta}(u)&=\sum_{j,k=1}^n\int_{\partial\Omega} |\nabla\rho|^{-1}\left(\bar Z_{k,\eta}u\right)\Upsilon^{\bar k j}\frac{\partial\rho}{\partial z_j}\overline{u}\, d\sigma,\text{ and}\\
    H_{\Upsilon,\eta}(u)&=\sum_{j,k=1}^n\int_{\partial\Omega} |\nabla\rho|^{-1}\frac{\partial}{\partial\bar z_k}\left(\Upsilon^{\bar k j}\frac{\partial\rho}{\partial z_j}\right)|u|^2\, d\sigma.
  \end{align*}

  We begin by using integration by parts to evaluate the second term of the right-hand side in \eqref{eq:G_Upsilon}.
  \begin{multline*}
    \sum_{j,k=1}^n\int_\Omega \bar Z_{j,\eta}^*u\Upsilon^{\bar k j}\overline{\bar Z_{k,\eta}^* u}\, dV=\\
    \sum_{j,k=1}^n\int_\Omega \bar Z_{k,\eta}\bar Z_{j,\eta}^*u\Upsilon^{\bar k j}\overline{u}\, dV+\sum_{j=1}^n\int_\Omega \bar Z_{j,\eta}^*u\Upsilon^{j}\overline{u}\, dV-E_{\Upsilon,\eta}(u).
  \end{multline*}
  We compute the commutator
  \[
    [\bar Z_{k,\eta},\bar Z_{j,\eta}^*]=-\eta\frac{\partial}{\partial\bar z_k}\overline{\Upsilon^{j}}-\eta\frac{\partial}{\partial z_j}\Upsilon^k.
  \]
  Substituting this and integrating by parts again, we obtain
  \begin{multline*}
    \sum_{j,k=1}^n\int_\Omega \bar Z_{j,\eta}^*u\Upsilon^{\bar k j}\overline{\bar Z_{k,\eta}^* u}\, dV=
    \sum_{j,k=1}^n\int_\Omega \bar Z_{k,\eta}u\Upsilon^{\bar k j}\overline{\bar Z_{j,\eta}u}\, dV+\sum_{k=1}^n\int_\Omega \bar Z_{k,\eta}u\overline{\Upsilon^{k}}\overline{u}\, dV\\
    +\sum_{j=1}^n\int_\Omega \bar Z_{j,\eta}^*u\Upsilon^{j}\overline{u}\, dV-\sum_{j,k=1}^n\int_\Omega\eta \left(\frac{\partial}{\partial\bar z_k}\overline{\Upsilon^{j}}+\frac{\partial}{\partial z_j}\Upsilon^k\right)\Upsilon^{\bar k j}|u|^2\, dV\\
    -E_{\Upsilon,\eta}(u)-F_{\Upsilon,\eta}(u).
  \end{multline*}
  Hence,
  \begin{multline*}
    G_{\Upsilon,\eta}(u,u)=
    \sum_{j=1}^n\norm{\bar Z_{j,\eta}u}^2_{L^2(\Omega)}+\sum_{k=1}^n\int_\Omega \bar Z_{k,\eta}u\overline{\Upsilon^{k}}\overline{u}\, dV+\sum_{j=1}^n\int_\Omega \bar Z_{j,\eta}^*u\Upsilon^{j}\overline{u}\, dV\\
    -\sum_{j,k=1}^n\int_\Omega\eta \left(\frac{\partial}{\partial\bar z_k}\overline{\Upsilon^{j}}+\frac{\partial}{\partial z_j}\Upsilon^k\right)\Upsilon^{\bar k j}|u|^2\, dV
    -E_{\Upsilon,\eta}(u)-F_{\Upsilon,\eta}(u).
  \end{multline*}
  Integration by parts gives us
  \begin{multline*}
    \sum_{j=1}^n\int_\Omega \bar Z_{j,\eta}^*u\Upsilon^{j}\overline{u}\, dV=\sum_{j=1}^n\int_\Omega u\Upsilon^{j}\overline{\bar Z_{j,\eta}u}\, dV+\sum_{j=1}^n\int_\Omega \frac{\partial}{\partial z_j}\Upsilon^{j}|u|^2\, dV\\
    -\sum_{j=1}^n\int_{\partial\Omega}|\nabla\rho|^{-1}\frac{\partial\rho}{\partial z_j}\Upsilon^{j}|u|^2\, d\sigma,
  \end{multline*}
  so since $\sum_{j=1}^n\frac{\partial\rho}{\partial z_j}\Upsilon^j=\sum_{j,k=1}^n\left(\frac{\partial}{\partial\bar z_k}\left(\frac{\partial\rho}{\partial z_j}\Upsilon^{\bar k j}\right)-\frac{\partial^2\rho}{\partial z_j\partial\bar z_k}\Upsilon^{\bar k j}\right)$, we have
  \begin{multline*}
    G_{\Upsilon,\eta}(u,u)=
    \sum_{j=1}^n\norm{\bar Z_{j,\eta}u}^2_{L^2(\Omega)}+\sum_{j=1}^n\int_\Omega 2\re\left(\bar Z_{j,\eta}u\overline{\Upsilon^{j}}\overline{u}\right)\, dV\\
    -\sum_{j,k=1}^n\int_\Omega\eta \left(\frac{\partial}{\partial\bar z_k}\overline{\Upsilon^{j}}+\frac{\partial}{\partial z_j}\Upsilon^k\right)\Upsilon^{\bar k j}|u|^2\, dV
    +\sum_{j=1}^n\int_\Omega \frac{\partial}{\partial z_j}\Upsilon^{j}|u|^2\, dV\\
    +\sum_{j,k=1}^n\int_{\partial\Omega} |\nabla\rho|^{-1}\frac{\partial^2\rho}{\partial z_j\partial\bar z_k}\Upsilon^{\bar k j}|u|^2\, d\sigma
    -E_{\Upsilon,\eta}(u)-F_{\Upsilon,\eta}(u)-H_{\Upsilon,\eta}(u).
  \end{multline*}
  Note that
  \begin{multline*}
    \sum_{j=1}^n\norm{\bar Z_{j,\eta} u}^2_{L^2(\Omega)}+\sum_{j=1}^n\int_\Omega 2\re\left(\bar Z_{j,\eta} u\overline{\Upsilon^{j}}\overline{u}\right)\, dV=\\
    \sum_{j=1}^n\norm{\bar Z_{j,\eta-1}u}^2_{L^2(\Omega)}-\sum_{j=1}^n\norm{\Upsilon^j u}^2_{L^2(\Omega)},
  \end{multline*}
  so
  \begin{multline*}
    G_{\Upsilon,\eta}(u,u)=
    \sum_{j=1}^n\norm{\bar Z_{j,\eta-1}u}^2_{L^2(\Omega)}-\sum_{j=1}^n\norm{\Upsilon^j u}^2_{L^2(\Omega)}\\
    -\sum_{j,k=1}^n\int_\Omega\eta \left(\frac{\partial}{\partial\bar z_k}\overline{\Upsilon^{j}}+\frac{\partial}{\partial z_j}\Upsilon^k\right)\Upsilon^{\bar k j}|u|^2\, dV
    +\sum_{j=1}^n\int_\Omega \frac{\partial}{\partial z_j}\Upsilon^{j}|u|^2\, dV\\
    +\sum_{j,k=1}^n\int_{\partial\Omega} |\nabla\rho|^{-1}\frac{\partial^2\rho}{\partial z_j\partial\bar z_k}\Upsilon^{\bar k j}|u|^2\, d\sigma
    -E_{\Upsilon,\eta}(u)-F_{\Upsilon,\eta}(u)-H_{\Upsilon,\eta}(u).
  \end{multline*}
  Expanding \eqref{eq:Theta_defined}, we find that
  \begin{equation}
  \label{eq:Theta_identity}
    \Theta_{\Upsilon,\eta}=\sum_{j=1}^n\frac{\partial}{\partial z_j}\Upsilon^{j}-\sum_{j,k=1}^n\eta\left(\frac{\partial}{\partial\bar z_k}\overline{\Upsilon^{j}}+\frac{\partial}{\partial z_j}\Upsilon^k\right)\Upsilon^{\bar k j}
    -\sum_{j=1}^n\abs{\Upsilon^j}^2,
  \end{equation}
  so we have
  \begin{multline*}
    G_{\Upsilon,\eta}(u,u)=
    \sum_{j=1}^n\norm{\bar Z_{j,\eta-1}u}^2_{L^2(\Omega)}+\int_\Omega \Theta_{\Upsilon,\eta}|u|^2\, dV\\
    +\sum_{j,k=1}^n\int_{\partial\Omega} |\nabla\rho|^{-1}\frac{\partial^2\rho}{\partial z_j\partial\bar z_k}\Upsilon^{\bar k j}|u|^2\, dV
    -E_{\Upsilon,\eta}(u)-F_{\Upsilon,\eta}(u)-H_{\Upsilon,\eta}(u).
  \end{multline*}
  whenever $q=0$, $u\in C^2(\overline\Omega)$, and the coefficients of $\Upsilon$ are in $C^2(U)$.

  Now suppose that $q=0$ and $u\in C^2(\overline\Omega)$, but $\Upsilon\in\mathcal{M}^2_{\Omega,\eta}(U)$.  Let $V\subset\subset U$ denote the interior of the support of $u$.  If we regularize the coefficients of $\Upsilon$ by convolution and restrict to $V$, we obtain a sequence $\{\Upsilon_\ell\}$ of $n\times n$ Hermitian matrices with coefficients in $C^\infty(V)$ that are uniformly bounded in $L^\infty(V)$ and converge in $W^{1,2}(V)$ and $W^{2,1}(V)$ to the corresponding coefficients of $\Upsilon$.  Observe that \eqref{eq:Upsilon_mixed_vanishes} implies that the trace of $\Upsilon$ in $L^2(\partial\Omega)$ satisfies
  \begin{equation}
  \label{eq:Upsilon_tangential}
    \sum_{k=1}^n\frac{\partial\rho}{\partial\bar z_k}\Upsilon^{\bar k j}\equiv 0\text{ for every }1\leq j\leq n,
  \end{equation}
  By the trace theorem for $W^{1,2}(V)$, $\sum_{k=1}^n\frac{\partial\rho}{\partial\bar z_k}\Upsilon_\ell^{\bar k j}\rightarrow 0$ in $L^2(\partial\Omega)$ for every $1\leq j\leq n$, so $E_{\Upsilon_\ell,\eta}(u)\rightarrow 0$ and $F_{\Upsilon_\ell,\eta}(u)\rightarrow 0$.  Here we have also used the fact that $\Upsilon_\ell^j$ is uniformly bounded in $L^2(V)$, since this term arises in $\bar Z_{k,\eta}$ and $\bar Z_{j,\eta}^*$.

  Observe that $\nabla\Upsilon \in L^1(\partial\Omega\cap V)$ and the combination of
  \eqref{eq:Upsilon_normal_vanishes} and (\ref{eq:Upsilon_mixed_vanishes}) imply that the trace of $\nabla \Upsilon$ satisfies
  \begin{equation}
  \label{eq:Upsilon_tangential_higher_order}
    \sum_{j,k=1}^n\frac{\partial\rho}{\partial\bar z_k}\nabla\Upsilon^{\bar k j}\frac{\partial\rho}{\partial z_j}=0.
  \end{equation}
  Since $\rho$ is $C^{1,1}$, $\nabla^2\rho$ exists almost everywhere.  For every $1\leq k\leq n$, $\frac{\partial}{\partial\bar z_k}-|\partial\rho|^{-2}\frac{\partial\rho}{\partial\bar z_k}\sum_{\ell=1}^n\frac{\partial\rho}{\partial z_\ell}\frac{\partial}{\partial\bar z_\ell}$ is a tangential operator on $\partial\Omega\cap V$, so at points where $\nabla^2\rho$ exists, we may differentiate \eqref{eq:Upsilon_tangential} by this operator and sum over $1\leq k\leq n$ to obtain
  \[
  0=\sum_{k=1}^n\left(\frac{\partial}{\partial\bar z_k}-|\partial\rho|^{-2}\frac{\partial\rho}{\partial\bar z_k}\sum_{\ell=1}^n\frac{\partial\rho}{\partial z_\ell}\frac{\partial}{\partial\bar z_\ell}\right)\left(\sum_{j=1}^n\Upsilon^{\bar k j}\frac{\partial\rho}{\partial z_j}\right).
  \]
  If we expand the second term in this derivative with the product rule and apply \eqref{eq:Upsilon_tangential} and \eqref{eq:Upsilon_tangential_higher_order}, we see that these terms vanish almost everywhere on the boundary, leaving us with
  \[
  0=\sum_{k=1}^n\frac{\partial}{\partial\bar z_k}\left(\sum_{j=1}^n\Upsilon^{\bar k j}\frac{\partial\rho}{\partial z_j}\right)\text{ almost everywhere on }\partial\Omega\cap V.
  \]
  By the trace theorem for $W^{2,1}(V)$, we see that $H_{\Upsilon_\ell,\eta}(u)\rightarrow 0$ as well.

  If we expand the quadratic forms in \eqref{eq:G_Upsilon}, observe that the two terms of the form $\sum_{j,k=1}^n\Upsilon^k\Upsilon^{\bar k j}\Upsilon^{\bar j}$ will cancel each other, so the remaining terms in \eqref{eq:G_Upsilon} with $\Upsilon$ will involve integrating a function in $L^\infty(U)$ against $\Upsilon^{\bar k j}$ for $1\leq j,k\leq n$, $\sum_{k=1}^n\Upsilon^k\Upsilon^{\bar k j}$ or $\Upsilon^j$ for $1\leq j\leq n$, or $\sum_{j=1}^n|\Upsilon^j|^2$.  In every case, convergence in $W^{1,2}(V)$ and uniform boundedness in $L^\infty(V)$ will suffice to take the limit.  This is also true for the final term in \eqref{eq:Theta_identity}.  Convergence in $W^{1,2}(V)$ and uniform boundedness in $L^\infty(V)$ will also allow us to take the limit with the remaining terms in \eqref{eq:Theta_identity}.  By the Trace Theorem, convergence in $W^{1,1}(V)$ guarantees convergence in $L^1(\partial\Omega\cap V)$, so we can take the limit in the boundary term in \eqref{eq:gradient_transformation}.  As a result, we may take a limit and obtain \eqref{eq:gradient_transformation} whenever $q=0$, $u\in C^2(\overline\Omega)$, and the coefficients of $\Upsilon$ are in $L^\infty(V)\cap W^{1,2}(V)\cap W^{2,1}(V)$.

  For $q>0$, we apply the $q=0$ identity to each coefficient $u_I$ and sum over all $I\in\mathcal{I}_q$ to obtain \eqref{eq:gradient_transformation}.  When $u\in W^{1,2}_{(0,q)}(\Omega)$, we can approximate $u$ in the $W^{1,2}(\Omega)$ norm by a sequence in $C^\infty_{(0,q)}(\overline\Omega)$ and take the limit.  Since $\Upsilon^j\in L^\infty(U)$ for all $1\leq j\leq n$ and $\Theta_\Upsilon\in L^\infty(U)$, all of the terms in this limit will converge, and we obtain \eqref{eq:gradient_transformation} whenever $u\in W^{1,2}_{(0,q)}(\Omega)$.

\end{proof}

Our first application of this key identity will be a sequence of lemmas which can be used to show that \eqref{eq:maximal_estimate} is independent of the choice of orthonormal coordinates and that left-hand side of \eqref{eq:maximal_estimate} is comparable to $G_{\Upsilon,\eta}$ when we have suitable bounds on the eigenvalues of $\Upsilon$.

We first show that we can always estimate the direction that is missing from our maximal estimates, provided that we multiply it by a suitable defining function.
\begin{lem}
\label{lem:bad_direction_estimate}
  Let $\Omega\subset\mathbb{C}^n$ be a domain with $C^{1,1}$ boundary.  Let $p\in\partial\Omega$, let $U$ be a neighborhood of $p$, and let $\rho\in C^{1,1}(U)\cap C^\infty(U\backslash\partial\Omega)$ be the defining function for $\Omega$ given by Lemma \ref{lem:rho_derivative_estimate}.  For every $\epsilon>0$, there exists a constant $C_\epsilon>0$ so that if $u\in W^{1,2}(\Omega)$ is supported in $\set{z\in \overline\Omega\cap U:-\frac{\sqrt{\epsilon}}{2}<\rho(z)\leq 0}$, then
  \begin{equation}
  \label{eq:bad_direction_estimate}
    \norm{(-\rho)Z u}^2_{L^2(\Omega)}\leq
    \sum_{j=1}^n\epsilon\norm{\frac{\partial u}{\partial\bar z_j}}^2_{L^2(\Omega)}+C_\epsilon\norm{u}^2_{L^2(\Omega)},
  \end{equation}
  where $Z=|\partial\rho|^{-1}\sum_{j=1}^n\frac{\partial\rho}{\partial\bar z_j}\frac{\partial}{\partial z_j}$.
\end{lem}

\begin{proof}
  For $1\leq j,k\leq n$, set $\Upsilon^{\bar k j}=(-\rho)^2|\partial\rho|^{-2}\frac{\partial\rho}{\partial z_k}\frac{\partial\rho}{\partial\bar z_j}$.  Using \eqref{eq:rho_derivative_estimate}, each $\Upsilon^{\bar k j}\in C^{2,1}(U)$, so $\Upsilon\in\mathcal{M}^2_{\Omega,0}(U)$.  Using Lemma \ref{lem:gradient_transformation} with $\eta=0$, \eqref{eq:gradient_transformation} implies that
  \begin{multline*}
    \sum_{j=1}^n\norm{\frac{\partial u}{\partial\bar z_j}}^2_{L^2(\Omega)}-\norm{(-\rho)\bar Z u}^2_{L^2(\Omega)}+\norm{(-\rho)Z u}^2_{L^2(\Omega)}=\\
    \sum_{j=1}^n\norm{\left(\frac{\partial}{\partial\bar z_j}+\Upsilon^j\right)u}^2_{L^2(\Omega)}+\int_\Omega \Theta_{\Upsilon,0}|u|^2\, dV.
  \end{multline*}
  Rearranging terms and substituting \eqref{eq:Theta_defined} gives us
  \begin{multline*}
    \norm{(-\rho)Z u}^2_{L^2(\Omega)}\leq
    \norm{(-\rho)\bar Z u}^2_{L^2(\Omega)}\\+\sum_{j=1}^n2\re\int_\Omega\frac{\partial u}{\partial\bar z_j}\overline{\Upsilon^j u}\, dV+\int_\Omega \sum_{j,k=1}^n\frac{\partial^2}{\partial z_j\partial\bar z_k}\Upsilon^{\bar k j}|u|^2\, dV.
  \end{multline*}
  Using the Cauchy-Schwarz inequality and the small constant/large constant inequality, we see that \eqref{eq:bad_direction_estimate} follows.

\end{proof}

Next, we show that estimates for tangential $(1,0)$-derivatives in any basis are comparable to estimates in our special basis.  In particular, this implies that Definition \ref{defn:maximal_estimate} is independent of the choice of orthonormal coordinates.
\begin{lem}
\label{lem:gradient_comparison}
  Let $\Omega\subset\mathbb{C}^n$ be a domain with $C^{1,1}$ boundary.  Let $p\in\partial\Omega$ and let $U$ be a neighborhood of $p$ that is sufficiently small so that the special orthonormal basis $\{\tilde L_j\}_{j=1}^n$ given by Lemma \ref{lem:orthonormal_coordinates} exists.  Let $\{L_j\}_{j=1}^n$ be an orthonormal basis for $T^{1,0}(U)$ with Lipschitz coefficients such that $\{L_j|_{\partial\Omega\cap U}\}_{j=1}^{n-1}$ is an orthonormal basis for $T^{1,0}(\partial\Omega\cap U)$.  For every $\epsilon>0$, there exists a neighborhood $U_\epsilon\subset U$ of $p$ and a constant $C_\epsilon>0$ such that
  \begin{equation}
  \label{eq:gradient_comparison_1}
    \sum_{j=1}^{n-1}\norm{\tilde L_j u}^2_{L^2(\Omega)}\leq\\
    \sum_{j=1}^{n-1}(1+\epsilon)\norm{L_j u}^2_{L^2(\Omega)}+\sum_{j=1}^n\epsilon\norm{\frac{\partial u}{\partial\bar z_j}}^2_{L^2(\Omega)}+C_\epsilon\norm{u}^2_{L^2(\Omega)}
  \end{equation}
  and
  \begin{equation}
  \label{eq:gradient_comparison_2}
    \sum_{j=1}^{n-1}\norm{L_j u}^2_{L^2(\Omega)}\leq\\
    \sum_{j=1}^{n-1}(1+\epsilon)\norm{\tilde L_j u}^2_{L^2(\Omega)}+\sum_{j=1}^n\epsilon\norm{\frac{\partial u}{\partial\bar z_j}}^2_{L^2(\Omega)}+C_\epsilon\norm{u}^2_{L^2(\Omega)}
  \end{equation}
  for all $u\in W^{1,2}(\Omega)$ supported in $\overline\Omega\cap U_\epsilon$.
\end{lem}

\begin{proof}
  Let $\rho\in C^{1,1}(U)\cap C^\infty(U\backslash\partial\Omega)$ be the defining function given by Lemma \ref{lem:rho_derivative_estimate}, and recall that $\tilde L_n=|\partial\rho|^{-1}\sum_{j=1}^n\frac{\partial\rho}{\partial\bar z_j}\frac{\partial}{\partial z_j}$.

  Since $\sum_{j=1}^{n}\norm{\tilde L_j u}^2_{L^2(\Omega)}=\sum_{j=1}^{n}\norm{L_j u}^2_{L^2(\Omega)}$, we have
  \begin{equation}
  \label{eq:coordinate_difference}
    \sum_{j=1}^{n-1}\left(\norm{\tilde L_j u}^2_{L^2(\Omega)}-\norm{L_j u}^2_{L^2(\Omega)}\right)=
    \norm{L_n u}^2_{L^2(\Omega)}-\norm{\tilde L_n u}^2_{L^2(\Omega)}.
  \end{equation}
  Since $\tilde L_n$ and $L_n$ are linearly dependent on $\partial\Omega$,
  \[
    \abs{\left<L_n,\tilde L_j\right>}\leq O(|\rho|)\text{ on }U\text{ for all }1\leq j\leq n-1.
  \]
  This means that
  \[
    1-\abs{\left<L_n,\tilde L_n\right>}^2=\abs{L_n-\left<L_n,\tilde L_n\right>\tilde L_n}^2\leq O(|\rho|^2)\text{ on }U,
  \]
  so
  \[
    1\geq\abs{\left<L_n,\tilde L_n\right>}^2\geq 1-O(|\rho|^2)\text{ on }U.
  \]
  Using the decomposition $L_n=\sum_{j=1}^n\left<L_n,\tilde L_j\right>\tilde L_j$, we see that
  \begin{multline*}
    \norm{L_n u}^2_{L^2(\Omega)}-\norm{\tilde L_n u}^2_{L^2(\Omega)}=\norm{\sum_{j=1}^{n-1}\left<L_n,\tilde L_j\right>\tilde L_j u}^2_{L^2(\Omega)}\\
    +\int_\Omega\sum_{j=1}^{n-1}\left<L_n,\tilde L_j\right>\tilde L_j u\overline{\left<L_n,\tilde L_n\right>\tilde L_n u} \, dV-\norm{\sqrt{1-\abs{\left<L_n,\tilde L_n\right>}^2}\tilde L_n u}^2_{L^2(\Omega)}
  \end{multline*}
  so there must exist a constant $C>0$ such that
  \begin{multline*}
    \abs{\norm{L_n u}^2_{L^2(\Omega)}-\norm{\tilde L_n u}^2_{L^2(\Omega)}}\leq\\
    \sum_{j=1}^{n-1} C\norm{|\rho|\tilde L_j u}^2_{L^2(\Omega)}+\sum_{j=1}^{n-1}C\norm{\tilde L_j u}_{L^2(\Omega)}\norm{|\rho|\tilde L_n u}_{L^2(\Omega)}+C\norm{|\rho|\tilde L_n u}^2_{L^2(\Omega)}.
  \end{multline*}
  By the small constant/large constant inequality, we see that for every $\epsilon>0$ there exists a constant $C_\epsilon>0$ and a neighborhood $U_\epsilon$ of $p$ so that we have
  \[
    \abs{\norm{L_n u}^2_{L^2(\Omega)}-\norm{\tilde L_n u}^2_{L^2(\Omega)}}\leq
    \sum_{j=1}^{n-1}\epsilon\norm{\tilde L_j u}^2_{L^2(\Omega)}+C_\epsilon\norm{|\rho|\tilde L_n u}^2_{L^2(\Omega)}
  \]
  whenever $u\in W^{1,2}(\Omega)$ is supported in $\overline\Omega\cap U_\epsilon$.  Now \eqref{eq:bad_direction_estimate} with $Z=\tilde L_n$ allows us to further refine $C_\epsilon$ and $U_\epsilon$ so that
  \begin{multline*}
    \abs{\norm{L_n u}^2_{L^2(\Omega)}-\norm{\tilde L_n u}^2_{L^2(\Omega)}}\leq\\
    \sum_{j=1}^{n-1}\epsilon\norm{\tilde L_j u}^2_{L^2(\Omega)}+\sum_{j=1}^n\epsilon\norm{\frac{\partial u}{\partial\bar z_j}}^2_{L^2(\Omega)}+C_\epsilon\norm{u}^2_{L^2(\Omega)}.
  \end{multline*}
  If we substitute \eqref{eq:coordinate_difference}, then we may subtract $\sum_{j=1}^{n-1}\epsilon\norm{\tilde L_j u}^2_{L^2(\Omega)}$ and make a final adjustment to $C_\epsilon$ and $U_\epsilon$ so that \eqref{eq:gradient_comparison_1} and \eqref{eq:gradient_comparison_2} follow.

\end{proof}

\begin{lem}
\label{lem:G_bounded_by_gradient}
  Let $\Omega\subset\mathbb{C}^n$ be a domain with $C^{1,1}$ boundary.  Let $p\in\partial\Omega$ and let $U$ be a neighborhood of $p$.  Let $0\leq\eta\leq 1$, and let $\Upsilon\in\mathcal{M}^1_\Omega(U)$.  For $0\leq q\leq n$ and $u,v\in W^{1,2}(\Omega)$ supported in $\overline\Omega\cap U$, define $G_{\Upsilon,\eta}(u,v)$ by \eqref{eq:G_Upsilon}. Let $\{L_j\}_{j=1}^n$ be an orthonormal basis for $T^{1,0}(U)$ with Lipschitz coefficients such that $\{L_j|_{\partial\Omega\cap U}\}_{j=1}^{n-1}$ is an orthonormal basis for $T^{1,0}(\partial\Omega\cap U)$.  Then for every $\epsilon>0$, there exists a neighborhood $U_\epsilon\subset U$ of $p$ and a constant $C_\epsilon>0$ such that
  \begin{equation}
  \label{eq:G_bounded_by_gradient}
    G_{\Upsilon,\eta}(u,u)\leq(1+\epsilon)\left(\sum_{j=1}^n\norm{\frac{\partial}{\partial\bar z_j}u}^2_{L^2(\Omega)}+\sum_{j=1}^{n-1}\norm{L_j u}^2_{L^2(\Omega)}\right)\\
    +C_\epsilon\norm{u}^2_{L^2(\Omega)}
  \end{equation}
  for all $u\in W^{1,2}(\Omega)$ supported in $\overline\Omega\cap U_\epsilon$.
\end{lem}

\begin{proof}
  Since every eigenvalue of $\Upsilon$ is non-negative and bounded by $1$,
  \begin{multline*}
    \sum_{j,k=1}^n\int_{\Omega\cap U} \left(\frac{\partial u}{\partial\bar z_k}-\eta\Upsilon^k u\right)(\delta_{j k}-\Upsilon^{\bar k j})\left(\overline{\frac{\partial u}{\partial\bar z_j}}-\eta\overline{\Upsilon^j u}\right)\, dV\leq\\
    \sum_{j=1}^n\norm{\frac{\partial u}{\partial\bar z_j}-\eta\Upsilon^j u}^2_{L^2(\Omega)},
  \end{multline*}
  so the Cauchy-Schwarz inequality followed by the small constant/large constant inequality imply that for every $\epsilon>0$ there exists $C_\epsilon>0$ such that
  \begin{multline}
  \label{eq:G_first_estimate}
    \sum_{j,k=1}^n\int_{\Omega\cap U} \left(\frac{\partial u}{\partial\bar z_k}-\eta\Upsilon^k u\right)(\delta_{j k}-\Upsilon^{\bar k j})\left(\overline{\frac{\partial u}{\partial\bar z_j}}-\eta\overline{\Upsilon^j u}\right)\, dV\leq\\
    \sum_{j=1}^n(1+\epsilon)\norm{\frac{\partial u}{\partial\bar z_j}}^2_{L^2(\Omega)}+C_\epsilon\norm{u}^2_{L^2(\Omega)}.
  \end{multline}
  Since \eqref{eq:G_first_estimate} gives us \eqref{eq:G_bounded_by_gradient} for the first term in \eqref{eq:G_Upsilon}, we may now turn our attention to the second term.

  Let $\rho\in C^{1,1}(U)\cap C^\infty(U\backslash\partial\Omega)$ be the defining function given by Lemma \ref{lem:rho_derivative_estimate}, and let $\{\tilde L_j\}_{j=1}^n$ be the orthonormal basis given by Lemma \ref{lem:orthonormal_coordinates} (after possibly shrinking $U$).  Since $\left<\frac{\partial}{\partial z_j},\tilde L_n\right>=|\partial\rho|^{-1}\sum_{j=1}^n\frac{\partial\rho}{\partial\bar z_j}$, we have
  \begin{equation}
  \label{eq:orthonormal_decomposition}
    \frac{\partial}{\partial z_j}=\sum_{k=1}^{n-1}\left<\frac{\partial}{\partial z_j},\tilde L_k\right>\tilde L_k+|\partial\rho|^{-1}\sum_{j=1}^n\frac{\partial\rho}{\partial\bar z_j}\tilde L_n\text{ for all }1\leq j\leq n.
  \end{equation}
  Using this decomposition with \eqref{eq:Upsilon_normal_vanishes} and \eqref{eq:Upsilon_mixed_vanishes} together with the fact that every eigenvalue of $\Upsilon$ is bounded above by one, we see that for every $\epsilon>0$ there exists $C_\epsilon>0$ such that
  \begin{multline*}
    \sum_{j,k=1}^n\int_{\Omega\cap U} \left(\frac{\partial u}{\partial z_j}+\eta\overline{\Upsilon^j} u\right)\Upsilon^{\bar k j}\left(\overline{\frac{\partial u}{\partial z_k}}+\eta\Upsilon^k u\right)\, dV\leq\\
    \sum_{j=1}^{n-1}(1+\epsilon)\norm{\tilde L_j u}^2_{L^2(\Omega)}+C_\epsilon\left(\norm{(-\rho)\tilde L_n u}^2_{L^2(\Omega)}+\norm{u}^2_{L^2(\Omega)}\right).
  \end{multline*}
  If we combine this with \eqref{eq:gradient_comparison_1}, \eqref{eq:bad_direction_estimate}, and \eqref{eq:G_first_estimate}, we see that for each $\epsilon>0$ we may further shrink $U_\epsilon$ and choose a larger constant $C_\epsilon>0$ so that \eqref{eq:G_bounded_by_gradient} holds.

\end{proof}

\begin{lem}
\label{lem:gradient_bounded_by_G}
  Let $\Omega\subset\mathbb{C}^n$ be a domain with $C^{1,1}$ boundary.  Let $p\in\partial\Omega$ and let $U$ be a neighborhood of $p$.  Let $0\leq\eta\leq 1$, and for $0<a<b<1$, let $\Upsilon\in\mathcal{M}^1_\Omega(U)$ have $n-1$ eigenvalues (counting multiplicity) in the interval $[a,b]$ almost everywhere on $U$.  For $0\leq q\leq n$ and $u,v\in W^{1,2}(\Omega)$ supported in $\overline\Omega\cap U$, define $G_{\Upsilon,\eta}(u,v)$ by \eqref{eq:G_Upsilon}. Let $\{L_j\}_{j=1}^n$ be an orthonormal basis for $T^{1,0}(U)$ with Lipschitz coefficients such that $\{L_j|_{\partial\Omega\cap U}\}_{j=1}^{n-1}$ is an orthonormal basis for $T^{1,0}(\partial\Omega\cap U)$.  Then for every $\epsilon>0$, there exists a neighborhood $U_\epsilon\subset U$ of $p$ and a constant $C_\epsilon>0$ such that
  \begin{multline}
  \label{eq:gradient_bounded_by_G}
    \sum_{j=1}^n\norm{\frac{\partial}{\partial\bar z_j}u}^2_{L^2(\Omega)}+\sum_{j=1}^{n-1}\norm{L_j u}^2_{L^2(\Omega)}\leq\\
    (1+\epsilon)\max\set{\frac{1}{a},\frac{1}{1-b}}G_{\Upsilon,\eta}(u,u)
    +C_\epsilon\norm{u}^2_{L^2(\Omega)}
  \end{multline}
  for all $u\in W^{1,2}(\Omega)$ supported in $\overline\Omega\cap U_\epsilon$.
\end{lem}

\begin{proof}
  Since every eigenvalue of $\Upsilon$ is bounded above by $b$,
  \begin{multline*}
    \sum_{j,k=1}^n\int_{\Omega\cap U} \left(\frac{\partial u}{\partial\bar z_k}-\eta\Upsilon^k u\right)(\delta_{j k}-\Upsilon^{\bar k j})\left(\overline{\frac{\partial u}{\partial\bar z_j}}-\eta\overline{\Upsilon^j u}\right)\, dV\geq\\
    \sum_{j=1}^n(1-b)\norm{\frac{\partial u}{\partial\bar z_j}-\eta\Upsilon^j u}^2_{L^2(\Omega)},
  \end{multline*}
  so the Cauchy-Schwarz inequality followed by the small constant/large constant inequality imply that for every $\epsilon>0$ there exists $C_\epsilon>0$ such that
  \begin{multline}
  \label{eq:G_first_estimate_b}
    \sum_{j,k=1}^n\int_{\Omega\cap U} \left(\frac{\partial u}{\partial\bar z_k}-\eta\Upsilon^k u\right)(\delta_{j k}-\Upsilon^{\bar k j})\left(\overline{\frac{\partial u}{\partial\bar z_j}}-\eta\overline{\Upsilon^j u}\right)\, dV\geq\\
    \sum_{j=1}^n(1-b)(1-\epsilon)\norm{\frac{\partial u}{\partial\bar z_j}}^2_{L^2(\Omega)}-C_\epsilon\norm{u}^2_{L^2(\Omega)}.
  \end{multline}
  Since \eqref{eq:G_first_estimate_b} gives us \eqref{eq:gradient_bounded_by_G} for the first term in \eqref{eq:G_Upsilon}, we may now turn our attention to the second term.

  Let $\rho\in C^{1,1}(U)\cap C^\infty(U\backslash\partial\Omega)$ be the defining function given by Lemma \ref{lem:rho_derivative_estimate}, and let $\{\tilde L_j\}_{j=1}^n$ be the orthonormal basis given by Lemma \ref{lem:orthonormal_coordinates} (after possibly shrinking $U$).  On $\partial\Omega\cap U$, $\tilde L_n$ spans the kernel of $\Upsilon$ by \eqref{eq:Upsilon_mixed_vanishes}, so the restriction of $\Upsilon$ to  $\Span\{\tilde L_j\}_{j=1}^{n-1}$ must have eigenvalues bounded below by $a$ almost everywhere on $\partial\Omega\cap U$.  If we use this with \eqref{eq:orthonormal_decomposition}, \eqref{eq:Upsilon_normal_vanishes}, and \eqref{eq:Upsilon_mixed_vanishes}, we see that for every $\epsilon>0$ there exists $C_\epsilon>0$ such that
  \begin{multline*}
    \sum_{j,k=1}^n\int_{\Omega\cap U} \left(\frac{\partial u}{\partial z_j}+\eta\overline{\Upsilon^j} u\right)\Upsilon^{\bar k j}\left(\overline{\frac{\partial u}{\partial z_k}}+\eta\Upsilon^k u\right)\, dV\geq\\
    \sum_{j=1}^{n-1}a(1-\epsilon)\norm{\tilde L_j u}^2_{L^2(\Omega)}-C_\epsilon\left(\norm{(-\rho)\tilde L_n u}^2_{L^2(\Omega)}+\norm{u}^2_{L^2(\Omega)}\right).
  \end{multline*}
  If we combine this with \eqref{eq:gradient_comparison_2}, \eqref{eq:bad_direction_estimate}, and \eqref{eq:G_first_estimate_b}, we see that for each $\epsilon>0$ we may further shrink $U_\epsilon$ and choose a larger constant $C_\epsilon>0$ so that \eqref{eq:G_bounded_by_gradient} holds.

\end{proof}

\section{A Necessary Condition for Maximal Estimates}
\label{sec:necessary_condition}

\begin{proof}[Proof of Theorem \ref{thm:necessary_condition}]
  We follow the proof of Theorem 3.2.1 in \cite{Hor65}.  Fix $\tilde p\in\partial\Omega\cap U$.  After a translation and rotation, we may assume that $\tilde p=0$ and that there exists a neighborhood $\tilde U$ of $\tilde p$ on which
  \[
    \Omega\cap\tilde U=\set{z\in\tilde U:y_n>\varphi(z',x_n)},
  \]
  where $z'=(z_1,\ldots,z_{n-1})$, $z_n=x_n+iy_n$, and $\varphi$ is a $C^{2,\alpha}$ function in some neighborhood of the origin that vanishes to second order at the origin.  It is easy to check that \eqref{eq:necessary_condition} is independent of our choice of defining function, so we let $\rho(z)=\varphi(z',\re z_n)-\im z_n$ on $\tilde U$.  Since $\partial\rho(0)=\frac{i}{2}dz_n$, we have
  \begin{equation}
  \label{eq:rho_normalization}
    |\nabla\rho(0)|=\frac{1}{\sqrt{2}}.
  \end{equation}

  Define
  \begin{equation}
  \label{eq:L_defn_1}
    \mathcal{L}(z'):=\sum_{j,k=1}^{n-1}\frac{\partial^2\varphi}{\partial z_j\partial\bar z_k}(0)z_j\bar z_k.
  \end{equation}
  After a unitary change of coordinates, we may assume that
  \begin{equation}
  \label{eq:L_defn_2}
    \mathcal{L}(z')=\sum_{j=1}^{n-1}\lambda_j|z_j|^2
  \end{equation}
  for some increasing sequence of real numbers $\{\lambda_j\}_{1\leq j\leq n-1}$.  By construction, each $\lambda_j$ represents an eigenvalue of the Levi-form at $\tilde p=0$ with respect to the defining function $\rho$.

  After further shrinking $\tilde U$, we let $\{L_j\}_{j=1}^n$ be the orthonormal coordinates given by Lemma \ref{lem:orthonormal_coordinates}.  Without loss of generality, we may assume that $L_j|_{0}=\frac{\partial}{\partial z_j}$ for all $1\leq j\leq n-1$.  We write $L_j=\sum_{k=1}^n u_j^k(z)\frac{\partial}{\partial z_k}$ for all $1\leq j\leq n$, where $u_j^k\in C^{1}(\tilde U)\cap W^{2,1}(\tilde U)$ for all $1\leq j,k\leq n$ and $(u_j^k(z))_{1\leq j,k\leq n}$ is a unitary matrix for every $z\in\tilde U$.  Define $\Upsilon(z)=(\Upsilon^{\bar k j}(z))_{1\leq j,k\leq n}$ by $\Upsilon^{\bar k j}(z)=\sum_{\{1\leq\ell\leq n-1:\lambda_\ell(0)>0\}}\overline{u_\ell^k(z)}u_\ell^j(z)$ for all $1\leq j,k\leq n$ and $z\in\tilde U$.  Note that $\Upsilon$ is Hermitian and every eigenvalue of $\Upsilon$ must be either $0$ or $1$.  If $\tilde\rho$ is the defining function given by Lemma \ref{lem:rho_derivative_estimate}, then $\sum_{k=1}^n\frac{\partial\tilde\rho}{\partial\bar z_k}\Upsilon^{\bar k j}\equiv 0$ on $\tilde U$, so \eqref{eq:Upsilon_normal_vanishes} must hold for any $C^{2,\alpha}$ defining function $\rho$.  Since $C^1(\tilde U)\cap W^{2,1}(\tilde U)$ is a Banach algebra, $\Upsilon^{\bar k j}\in C^1(\tilde U)\cap W^{2,1}(\tilde U)$.  We immediately obtain $\Upsilon^j\in L^\infty(\tilde U)$ for all $1\leq j\leq n$, so $\Upsilon\in\mathcal{M}^1_\Omega(\tilde U)$.  Since every eigenvalue of $\Upsilon$ is equal to either zero or one, $\Upsilon$ is a projection, which means $\Upsilon^2 = \Upsilon$.  This means that the components of $\Upsilon$ satisfy $\Upsilon^{\bar k j}=\sum_{\ell=1}^n\Upsilon^{\bar k\ell}\Upsilon^{\bar\ell j}$ on $\tilde U$, and hence $\Theta_{\Upsilon,1}\in L^\infty(\tilde U)$ by \eqref{eq:Theta_defined}.  We conclude that $\Upsilon\in\mathcal{M}^2_{\Omega,1}(\tilde U)$.  Observe that $\Upsilon$ is diagonal at the origin and for $1\leq j\leq n-1$ we have
  \begin{equation}
  \label{eq:positive_Upsilon_characterized}
    \Upsilon^{\bar j j}(0)=\begin{cases}1&\lambda_j>0\\0&\lambda_j\leq 0\end{cases}.
  \end{equation}

  Let $u\in C^1_{0,q}(\overline\Omega)\cap\dom\dbar^*$ be supported in $\overline\Omega\cap\tilde U$.  If we apply \eqref{eq:G_bounded_by_gradient} to each component of $u$ and combine this with \eqref{eq:maximal_estimate}, we see that for every $\tilde A>A$ there exists $C_{\tilde A}>0$ such that
  \begin{equation}
  \label{eq:Upsilon_maximal_estimate}
    G_{\Upsilon,1}(u,u)\leq
     \tilde A\left(\norm{\dbar u}^2_{L^2(\Omega)}+\norm{\dbar^* u}^2_{L^2(\Omega)}\right)+C_{\tilde A}\norm{u}^2_{L^2(\Omega)}.
  \end{equation}
  We note that this may require a further shrinking of $\tilde U$.  Using \eqref{eq:gradient_transformation} with $\eta=1$, we have
  \begin{multline*}
    \sum_{j=1}^n\sum_{J\in\mathcal{I}_q}\norm{\frac{\partial}{\partial\bar z_j}u_J}^2_{L^2(\Omega)}+\sum_{j,k=1}^n\int_{\partial\Omega} |\nabla\rho|^{-1}\rho_{j\bar k}\Upsilon^{\bar k j}|u|^2\, d\sigma\\
    \leq G_{\Upsilon,1}(u,u)+O(\norm{u}^2_{L^2(\Omega)}).
  \end{multline*}
  We substitute \eqref{eq:Upsilon_maximal_estimate} to obtain
  \begin{multline*}
    \sum_{j=1}^n\sum_{J\in\mathcal{I}_q}\norm{\frac{\partial}{\partial\bar z_j}u_J}^2_{L^2(\Omega)}+\sum_{j,k=1}^n\int_{\partial\Omega} |\nabla\rho|^{-1}\rho_{j\bar k}\Upsilon^{\bar k j}|u|^2\, d\sigma\\
    \leq \tilde{A}(\norm{\dbar u}^2_{L^2(\Omega)}+\norm{\dbar^* u}^2_{L^2(\Omega)})+\tilde B\norm{u}^2_{L^2(\Omega)}
  \end{multline*}
  for some $\tilde B>0$.  Substituting \eqref{eq:Morrey_Kohn} and rearranging terms gives us
  \begin{multline}
  \label{eq:Upsilon_inequality_u}
    \sum_{j,k=1}^n\int_{\partial\Omega}|\nabla\rho|^{-1}\rho_{j\bar k}\Upsilon^{\bar k j}|u|^2\, d\sigma-\tilde{A}\sum_{j,k=1}^n\sum_{I\in\mathcal{I}_{q-1}}\int_{\partial\Omega}u_{jI}|\nabla\rho|^{-1}\rho_{j\bar k}\overline{u_{kI}}\, d\sigma\\
    \leq (\tilde A-1)\sum_{j=1}^n\sum_{J\in\mathcal{I}_q}\norm{\frac{\partial}{\partial\bar z_j}u_J}^2_{L^2(\Omega)}+\tilde B\norm{u}^2_{L^2(\Omega)}.
  \end{multline}

  Let $\psi_1\in C^\infty_0(\mathbb{C}^{n-1})$ and $\psi_3\in C^\infty_0(\mathbb{R})$ satisfy $\psi_3\equiv 1$ in a neighborhood of $0$ and $\int_{\mathbb{R}}|\psi_3|^2=1$.  For $x+iy\in\C$, if
  we define
  \[
    \psi_2(x+iy)=\psi_3(y)\left(\psi_3(x)+iy\psi_3'(x)\right),
  \]
  then $\psi_2(z)$ is a smooth, compactly supported function on $\mathbb{C}$ satisfying
  \begin{equation}
  \label{eq:psi_2_derivative}
    \frac{\partial}{\partial\bar z}\psi_2(z)\Big|_{\im z=0}=0
  \end{equation}
  and
  \begin{equation}
  \label{eq:psi_2_integral}
    \int_{\mathbb{R}}\big|\psi_2(x)\big|^2dx=1.
  \end{equation}

  Let $f(z)$ be the holomorphic polynomial
  \[
    f(z)=\sum_{j,k=1}^{n-1}\frac{\partial^2\varphi}{\partial z_j\partial z_k}(0)z_j z_k+\sum_{j=1}^{n-1}2\frac{\partial^2\varphi}{\partial z_j\partial x_n}(0)z_j z_n+\frac{1}{2}\frac{\partial^2\varphi}{\partial x_n^2}(0)z_n^2.
  \]
  Then we have
  \begin{equation}
  \label{eq:A_Taylor_Series}
    \abs{\varphi(z',x_n)-\re f(z',x_n)-
    \sum_{j,k=1}^{n-1}\frac{\partial^2\varphi}{\partial z_j\partial\bar z_k}(0)z_j\bar z_k}\leq O(|z'|^{2+\alpha}+|x_n|^{2+\alpha}).
  \end{equation}

  For any $\tau>0$ we define a form in $C^\infty_{(0,q)}(\overline\Omega)\cap\dom\dbar^*$ by
  \[
    u^\tau(z)=\psi_1(\tau z')\psi_2(\tau z_n)e^{\tau^2(f(z)+i z_n)}\bigwedge_{j=1}^q\left(d\bar z_j-\left(\frac{\partial\rho}{\partial z_n}\right)^{-1}\frac{\partial\rho}{\partial z_j}\, d\bar z_n\right).
  \]
  We introduce the change of coordinates $z_j(\tau)=\tau^{-1}w_j$ for $1\leq j\leq n-1$ and $z_n(\tau)=\tau^{-1}\re w_n+i\tau^{-2}\im w_n$.  Using \eqref{eq:A_Taylor_Series}, we have
  \begin{multline}
  \label{eq:defining_function_limit}
    \lim_{\tau\rightarrow\infty}\tau^2(\varphi(z'(\tau),x_n(\tau))-y_n(\tau))=\\
    \re f(w',\re w_n)+\sum_{j,k=1}^{n-1}\frac{\partial^2\varphi}{\partial z_j\partial\bar z_k}(0)w_j\bar w_k-\im w_n,
  \end{multline}
  so as $\tau\rightarrow\infty$ in our special coordinates, we will be working on the domain
  \begin{equation}
  \label{eq:Omega_w}
    \Omega_w=\set{w\in\mathbb{C}^n:\im w_n>\sum_{j,k=1}^{n-1}\frac{\partial^2\varphi}{\partial z_j\partial\bar z_k}(0)w_j\bar w_k\\+\re f(w',\re w_n)}.
  \end{equation}
  We may compute
  \begin{equation}
  \label{eq:pointwise_limit}
    \lim_{\tau\rightarrow\infty}\abs{u^\tau(z(\tau))}^2 =\abs{\psi_1(w')}^2\abs{\psi_2(\re w_n)}^2
    e^{2\re f(w',\re w_n)-2\im w_n},
  \end{equation}
  so
  \begin{equation}
  \label{eq:L2_limit_draft}
    \lim_{\tau\rightarrow\infty}\tau^{2n+1}\norm{u^\tau}^2_{L^2(\Om)}=\int_{\Omega_w}\abs{\psi_1(w')}^2\abs{\psi_2(\re w_n)}^2
    e^{2\re f(w',\re w_n)-2\im w_n}\, dV_w.
  \end{equation}
  We compute $\int_a^\infty e^{-2x}\,dx=\frac{1}{2}e^{-2a}$ for any $a\in\mathbb{R}$.  Since $i\, dz_n\wedge d\bar z_n=2\, dx_n\wedge dy_n$, we have $dV_w=2\, dV_{w'}\wedge dx_n\wedge dy_n$.  With these facts in mind, we may use \eqref{eq:Omega_w} to evaluate the integral in \eqref{eq:L2_limit_draft} with respect to $\im w_n$ (keeping in mind that the integration is over $\Om_w$) and then use \eqref{eq:psi_2_integral} to evaluate with respect to $\re w_n$ and obtain
  \begin{equation}
  \label{eq:L2_limit}
    \lim_{\tau\rightarrow\infty}\tau^{2n+1}\norm{u^\tau}^2_{L^2(\Om)}=
    \int_{\mathbb{C}^{n-1}}\abs{\psi_1(w')}^2
    e^{-2\mathcal{L}(w')}\, dV_{w'}.
  \end{equation}
  Observe that $\, d\sigma|_{0}=\sqrt{2}\, dV_{z'}\wedge dx_n$.  Hence, we may use \eqref{eq:Omega_w} and \eqref{eq:psi_2_integral} to integrate \eqref{eq:pointwise_limit} over the boundary and obtain
  \begin{equation}
  \label{eq:boundary_limit}
    \lim_{\tau\rightarrow\infty}\tau^{2n-1}\int_{\partial\Omega}\abs{u^\tau}^2 \, d\sigma=\sqrt{2}\int_{\mathbb{C}^{n-1}}\abs{\psi_1(w')}^2
    e^{-2\mathcal{L}(w')}\, dV_{w'}.
  \end{equation}
  Similarly, using \eqref{eq:rho_normalization}, we also have
  \begin{multline}
  \label{eq:boundary_limit_hessian}
    \lim_{\tau\rightarrow\infty}\tau^{2n-1}\int_{\partial\Omega}\sum_{j,k=1}^n
    \sum_{K\in\mathcal{I}_{q-1}}|\nabla\rho|^{-1}\frac{\partial^2\rho}{\partial z_j\partial\bar z_k}u^\tau_{jK}\overline{u^\tau_{kK}}\, d\sigma=\\
    2\int_{\mathbb{C}^{n-1}}\abs{\psi_1(w')}^2\sum_{j=1}^q\frac{\partial^2\varphi}{\partial z_j\partial\bar z_j}(0)
    e^{-2\mathcal{L}(w')}dw'.
  \end{multline}
  Substituting \eqref{eq:L_defn_1} and \eqref{eq:L_defn_2} in \eqref{eq:boundary_limit_hessian}, we obtain
  \begin{multline}
  \label{eq:combined_limit}
    \lim_{\tau\rightarrow\infty}\tau^{2n-1}\int_{\partial\Omega}\sum_{j,k=1}^n
    \sum_{K\in\mathcal{I}_{q-1}}|\nabla\rho|^{-1}\frac{\partial^2\rho}{\partial z_j\partial\bar z_k}u^\tau_{jK}\overline{u^\tau_{kK}}\, d\sigma=\\
    2\int_{\mathbb{C}^{n-1}}\abs{\psi_1(w')}^2\sum_{j=1}^q\lambda_j
    e^{-2\mathcal{L}(w')}\, dV_{w'}.
  \end{multline}

  For $1\leq k\leq n-1$, we compute
  \begin{multline*}
    \frac{\partial u^\tau}{\partial\bar z_k}(z)=
    \tau\frac{\partial \psi_1}{\partial\bar z_k}(\tau z')\psi_2(\tau z_n)e^{\tau^2(f(z)+i z_n)}\bigwedge_{j=1}^q\left(d\bar z_j-\left(\frac{\partial\rho}{\partial z_n}\right)^{-1}\frac{\partial\rho}{\partial z_j}\, d\bar z_n\right)\\
    +\psi_1(\tau z')\psi_2(\tau z_n)e^{\tau^2(f(z)+i z_n)}\frac{\partial}{\partial\bar z_k}\bigwedge_{j=1}^q\left(d\bar z_j-\left(\frac{\partial\rho}{\partial z_n}\right)^{-1}\frac{\partial\rho}{\partial z_j}\, d\bar z_n\right).
  \end{multline*}
  Furthermore,
  \begin{multline}
  \label{eq:f_tau_derivative_n}
    \frac{\partial u^\tau}{\partial\bar z_n}(z)=
    \tau\psi_1(\tau z')\frac{\partial\psi_2}{\partial\bar z_n}(\tau z_n)e^{\tau^2(f(z)+i z_n)}\bigwedge_{j=1}^q\left(d\bar z_j-\left(\frac{\partial\rho}{\partial z_n}\right)^{-1}\frac{\partial\rho}{\partial z_j}\, d\bar z_n\right)\\
    +\psi_1(\tau z')\psi_2(\tau z_n)e^{\tau^2(f(z)+i z_n)}\frac{\partial}{\partial\bar z_n}\bigwedge_{j=1}^q\left(d\bar z_j-\left(\frac{\partial\rho}{\partial z_n}\right)^{-1}\frac{\partial\rho}{\partial z_j}\, d\bar z_n\right).
  \end{multline}
  Hence, using \eqref{eq:psi_2_derivative} and observing that the second term in each derivative is uniformly bounded in $\tau$, we have
  \begin{multline}
  \label{eq:pointwise_gradient_limit}
    \lim_{\tau\rightarrow\infty}\tau^{-2}\sum_{j=1}^n\abs{\frac{\partial}{\partial\bar z_j}u^\tau(z)\Big|_{z=z(\tau)}}^2 \\ =\sum_{j=1}^{n-1}\abs{\frac{\partial}{\partial\bar w_j}\psi_1(w')}^2\abs{\psi_2(\re w_n)}^2
    e^{2\re f(w',\re w_n)-2\im w_n}.
  \end{multline}
  Integrating \eqref{eq:pointwise_gradient_limit} as before, we obtain
  \begin{equation}
  \label{eq:gradient_limit}
    \lim_{\tau\rightarrow\infty}\tau^{2n-1}\sum_{j=1}^n\norm{\frac{\partial}{\partial\bar z_j}u^\tau}^2_{L^2(\Om)}=
    \int_{\mathbb{C}^{n-1}}\sum_{j=1}^{n-1}\abs{\frac{\partial}{\partial\bar w_j}\psi_1(w')}^2
    e^{-2\mathcal{L}(w')}\, dV_{w'}.
  \end{equation}

  If we substitute $u^\tau$ in \eqref{eq:Upsilon_inequality_u}, multiply by $\tau^{2n-1}$, and take the limit as $\tau\rightarrow\infty$, we may use \eqref{eq:positive_Upsilon_characterized}, \eqref{eq:boundary_limit}, \eqref{eq:L_defn_1}, \eqref{eq:L_defn_2}, \eqref{eq:rho_normalization}, \eqref{eq:combined_limit}, and \eqref{eq:gradient_limit} to obtain
  \begin{multline}
  \label{eq:Upsilon_inequality_psi}
    2\int_{\mathbb{C}^{n-1}}\left(\sum_{\{1\leq j\leq n-1:\lambda_j>0\}} \lambda_j-\tilde A\sum_{j=1}^q\lambda_j
    \right)|\psi_1(w')|^2e^{-2\mathcal{L}(w')}\, dV_{w'}\\
    \leq (\tilde A-1)\int_{\mathbb{C}^{n-1}}\sum_{j=1}^{n-1}\abs{\frac{\partial}{\partial\bar w_j}\psi_1(w')}^2
    e^{-2\mathcal{L}(w')}\, dV_{w'}.
  \end{multline}
  Since this holds for all $\psi_1\in C^\infty_0(\mathbb{C}^{n-1})$, we may use Lemma 3.2.2 in \cite{Hor65} to show
  \[
    \sum_{\{1\leq j\leq n-1:\lambda_j<0\}}-\lambda_j\geq\frac{1}{\tilde A-1}\left(\sum_{\{1\leq j\leq n-1:\lambda_j>0\}} \lambda_j-\tilde A\sum_{j=1}^q\lambda_j
    \right).
  \]
  Multiplying by $\frac{\tilde A-1}{\tilde A}$ and rearranging terms, this is equivalent to
  \[
    \sum_{j=1}^q\lambda_j+\sum_{\{1\leq j\leq n-1:\lambda_j<0\}}|\lambda_j|\geq\frac{1}{\tilde A}\sum_{j=1}^{n-1} |\lambda_j|.
  \]
  Since this holds for every $\tilde A>A$, \eqref{eq:necessary_condition} follows.
\end{proof}

\begin{cor}
\label{cor:Z_q}
  Let $\Omega\subset\mathbb{C}^n$ be a domain with $C^{2,\alpha}$ boundary, $\alpha>0$, and let $1\leq q\leq n-1$.  Assume that $\Omega$ admits maximal estimates on $(0,q)$-forms near $p\in\partial\Omega$.  For any $\tilde p\in \partial\Omega$ sufficiently close to $p$, if any eigenvalue of the Levi-form is non-vanishing at $\tilde p$, then $\partial\Omega$ satisfies $Z(q)$ at $\tilde p$.
\end{cor}

\begin{proof}
  Suppose that the Levi-form at $\tilde p$ has at most $n-q-1$ positive eigenvalues.  Then \eqref{eq:necessary_condition} implies
  \[
    \sum_{\{q+1\leq j\leq n-1:\lambda_j<0\}}|\lambda_j|\geq\frac{1}{A}\sum_{j=1}^{n-1}|\lambda_j|.
  \]
  The right-hand side is strictly positive by hypothesis, and hence the left-hand side must also be positive.  This means that $\lambda_j<0$ for at least one $q+1\leq j\leq n-1$, and so the Levi-form has at least $q+1$ negative eigenvalues.  Hence, the Levi-form at $\tilde p$ has at least $n-q$ positive eigenvalues or at least $q+1$ negative eigenvalues.

\end{proof}

\section{A Sufficient Condition for Maximal Estimates}

\begin{proof}[Proof of Proposition \ref{prop:sufficient_condition}]
Fix a neighborhood $V$ of $p$ that is relatively compact in $U$, and choose $0<\epsilon<\frac{1}{2}-\frac{1}{A}$ such at least $n-1$ eigenvalues
 of $\Upsilon$ lie in
  \[
    \left[\frac{1}{A(1-2\epsilon)},1-\frac{1}{A(1-2\epsilon)}\right]
  \]
  on $\overline V$.  Let $u\in C^1_{0,q}(\overline\Omega)\cap\dom\dbar^*$ be supported in $\overline\Omega\cap V$.  We begin with the Morrey-Kohn identity \eqref{eq:Morrey_Kohn}.  Using Lemma \ref{lem:gradient_transformation}, we may expand $\bar Z_{j,\eta-1}u_J$ in \eqref{eq:gradient_transformation} and combine this with \eqref{eq:Morrey_Kohn} to obtain
  \begin{multline*}
    \norm{\dbar u}^2_{L^2(\Omega)}+\norm{\dbar^* u}^2_{L^2(\Omega)}+\int_\Omega\Theta_{\Upsilon,\eta}|u|^2\, dV
    +\sum_{j=1}^n\sum_{J\in\mathcal{I}_q}\norm{(1-\eta)\Upsilon^j u_J}^2_{L^2(\Omega)}=\\
    G_{\Upsilon,\eta}(u,u)-
    2\Rre\sum_{j=1}^n\sum_{J\in\mathcal{I}_q}\int_\Omega(1-\eta)\overline{\Upsilon^j}\frac{\partial u_J}{\partial\bar z_j}\overline{u_J}\, dV\\
    +\sum_{j,k=1}^n\sum_{I\in\mathcal{I}_{q-1}}\int_{\partial\Omega}u_{jI}|\nabla\rho|^{-1}\rho_{j\bar k}\overline{u_{kI}}\, d\sigma-\sum_{j,k=1}^n\int_{\partial\Omega}|\nabla\rho|^{-1}\rho_{j\bar k}\Upsilon^{\bar k j}|u|^2\, d\sigma.
  \end{multline*}
  Since at least $n-1$ of the eigenvalues of $\Upsilon$ lie in $\left[\frac{1}{A(1-2\epsilon)},1-\frac{1}{A(1-2\epsilon)}\right]$ on $\overline V$, we apply \eqref{eq:gradient_bounded_by_G} to each component of $u$ to see that there exist a constant $C_\epsilon>0$ and a neighborhood $V_\epsilon\subset V$ of $p$ such that
  \begin{multline*}
    (1-\epsilon)A\left(\norm{\dbar u}^2_{L^2(\Omega)}+\norm{\dbar^* u}^2_{L^2(\Omega)}\right)+C_\epsilon\norm{u}^2_{L^2(\Omega)}\geq\\
    -2\Rre\sum_{j=1}^n\sum_{J\in\mathcal{I}_q}\int_\Omega(1-\eta)\overline{\Upsilon^j}\frac{\partial u_J}{\partial\bar z_j}\overline{u_J}\, dV
    +\sum_{j=1}^n\sum_{J\in\mathcal{I}_q}\norm{\frac{\partial u_J}{\partial\bar z_j}}^2_{L^2(\Omega)}+\sum_{j=1}^{n-1}\sum_{J\in\mathcal{I}_q}\norm{L_j u_J}^2_{L^2(\Omega)}\\
    +\sum_{j,k=1}^n\sum_{I\in\mathcal{I}_{q-1}}A\int_{\partial\Omega}u_{jI}|\nabla\rho|^{-1}\rho_{j\bar k}\overline{u_{kI}}\, d\sigma-\sum_{j,k=1}^n A\int_{\partial\Omega}|\nabla\rho|^{-1}\rho_{j\bar k}\Upsilon^{\bar k j}|u|^2\, d\sigma
  \end{multline*}
  whenever $u\in C^1_{0,q}(\overline\Omega)\cap\dom\dbar^*$ is supported in $\overline\Omega\cap V_\epsilon$.  Using, e.g., Lemma 4.7 in \cite{Str10}, we see that \eqref{eq:weak_z_q} implies that the boundary term is non-negative.  A final application of the Cauchy-Schwarz inequality and the small constant/large constant inequality will imply \eqref{eq:maximal_estimate}.
\end{proof}

In the spirit of \cite{HaRa15}, we have the following obvious corollary:
\begin{cor}
\label{cor:sufficient_condition_C_2}
  Let $\Omega\subset\mathbb{C}^n$ be a domain with $C^3$ boundary and let $1\leq q\leq n-1$. Let $\rho$ be a $C^3$ defining function for $\Omega$. Assume that for some $p\in\partial\Omega$ there exists a neighborhood $U$ of $p$, a constant $A>2$, and a positive semi-definite Hermitian $n\times n$ matrix $\tilde\Upsilon$ with entries in $C^{1,1}(\partial\Omega\cap U)$ such that:
  \begin{enumerate}
    \item \eqref{eq:Upsilon_tangential} holds on $U\cap\partial\Omega$ for every $1\leq k\leq n$,
    \item counting multiplicity, $\tilde\Upsilon$ has $n-1$ eigenvalues in $(1/A,1-1/A)$ on $\partial\Omega\cap U$,
    \item if $\{\lambda_1,\ldots,\lambda_{n-1}\}$ denote the eigenvalues of the Levi-form (computed with $\rho$) arranged in increasing order, then \eqref{eq:weak_z_q} holds on $U\cap\partial\Omega$.
  \end{enumerate}
  Then $\Omega$ admits a maximal estimate on $(0,q)$-forms near $p$.
\end{cor}

\begin{proof}
  Let $\{L_j\}_{j=1}^n$ be the orthonormal coordinates given by Lemma \ref{lem:orthonormal_coordinates}.  If we express the coefficients of $\tilde\Upsilon$ with respect to these coordinates, it is not difficult to extend $\tilde\Upsilon$ to $\Upsilon$ with coefficients in $C^2(U)$ in such a way that $\Upsilon$ is positive semi-definite (for example, see the construction used in Lemma 2.11 in \cite{HaRa15} or Lemma 3.1 in \cite{HaRa19}).  Since $\Upsilon$ has entries in $C^2$, it is trivial to note that $\Upsilon^j\in L^\infty(U)$ for all $1\leq j\leq n$ and $\Theta_{\Upsilon,0}\in L^\infty(U)$.  We may clearly choose $U$ sufficiently small so that $n-1$ eigenvalues of $\Upsilon$ lie in $(1/A,1-1/A)$ on $U$, so the conditions of Proposition \ref{prop:sufficient_condition} are satisfied.
\end{proof}

We now have enough information to completely characterize the case in which the Levi-form is non-vanishing at a point.
\begin{proof}[Proof of Theorem \ref{thm:Z_q}]
  Suppose that $\Omega$ satisfies $Z(q)$ at $p$.  Let $\rho$ be a $C^3$ defining function for $\Omega$ on some neighborhood $U$ of $p$.  After a rotation and a translation, we may assume that $p=0$, $\frac{\partial\rho}{\partial z_j}(0)=0$ for all $1\leq j\leq n-1$, $\frac{\partial\rho}{\partial z_n}(0)=-\frac{i}{2}|\partial\rho(0)|$, $\frac{\partial^2\rho}{\partial z_j\partial\bar z_k}(0)=0$ whenever $1\leq j,k\leq n-1$ and $j\neq k$, and $\frac{\partial^2\rho}{\partial z_j\partial\bar z_j}(0)\leq\frac{\partial^2\rho}{\partial z_k\partial\bar z_k}(0)$ whenever $1\leq j\leq k\leq n-1$.  We assume that $U$ is sufficiently small so that $\frac{\partial\rho}{\partial z_n}\neq 0$ on $U$.

  Let $m$ denote the number of negative eigenvalues of the Levi-form at $p$, so that $\frac{\partial^2\rho}{\partial z_j\partial\bar z_j}(0)<0$ if and only if $1\leq j\leq m$.  Let $0<a<b<1$.  For $z\in U$, set $\Upsilon^{\bar j j}(z)=b$ whenever $1\leq j\leq m$, $\Upsilon^{\bar j j}(z)=a$ whenever $m+1\leq j\leq n-1$, $\Upsilon^{\bar k j}(z)=0$ whenever $1\leq j,k\leq n-1$ and $j\neq k$, $\Upsilon^{\bar n j}(z)=-b\left(\frac{\partial\rho}{\partial\bar z_n}(z)\right)^{-1}\frac{\partial\rho}{\partial\bar z_j}(z)$ for all $1\leq j\leq m$, $\Upsilon^{\bar k n}(z)=-b\left(\frac{\partial\rho}{\partial z_n}(z)\right)^{-1}\frac{\partial\rho}{\partial z_k}(z)$ for all $1\leq k\leq m$, $\Upsilon^{\bar n j}(z)=-a\left(\frac{\partial\rho}{\partial\bar z_n}(z)\right)^{-1}\frac{\partial\rho}{\partial\bar z_j}(z)$ for all $m+1\leq j\leq n-1$, $\Upsilon^{\bar k n}(z)=-a\left(\frac{\partial\rho}{\partial z_n}(z)\right)^{-1}\frac{\partial\rho}{\partial z_k}(z)$ for all $m+1\leq k\leq n-1$, and $\Upsilon^{\bar n n}(z)=\abs{\frac{\partial\rho}{\partial z_n}(z)}^{-2}\left(\sum_{j=1}^mb\abs{\frac{\partial\rho}{\partial z_j}(z)}^{2}+\sum_{j=m+1}^{n-1}a\abs{\frac{\partial\rho}{\partial z_j}(z)}^{2}\right)$.  By construction, $\Upsilon$ satisfies \eqref{eq:Upsilon_tangential}.  At $p$, $\Upsilon$ has $m$ eigenvalues equal to $b$ and $n-1-m$ eigenvalues equal to $a$, so if we choose $A>\max\{1/a,1/(1-b)\}$, then we may shrink $U$ so that $\Upsilon$ has $n-1$ eigenvalues in the interval $(1/A,1-1/A)$ on $U$.

  Let $\{\lambda_j\}_{j=1}^{n-1}$ denote the eigenvalues of the Levi-form with respect to $\rho$ arranged in non-decreasing order.  By construction, $\lambda_j(0)=\frac{\partial^2\rho}{\partial z_j\partial\bar z_j}(0)$ for all $1\leq j\leq n-1$.  Hence,
  \[
    \sum_{j=1}^q\lambda_j(0)-\sum_{j,k=1}^n\Upsilon^{\bar k j}(0)\frac{\partial^2\rho}{\partial z_j\partial\bar z_k}(0)=\sum_{j=1}^q\lambda_j(0)-\sum_{j=1}^m b\lambda_j(0)-\sum_{j=m+1}^{n-1}a\lambda_j(0).
  \]
  If the Levi-form has at least $q+1$ negative eigenvalues, then $m\geq q+1$, so
  \[
    \sum_{j=1}^q\lambda_j(0)-\sum_{j,k=1}^n\Upsilon^{\bar k j}(0)\frac{\partial^2\rho}{\partial z_j\partial\bar z_k}(0)=\sum_{j=1}^m(1-b)\lambda_j(0)-\sum_{j=q+1}^m \lambda_j(0)-\sum_{j=m+1}^{n-1}a\lambda_j(0).
  \]
  Since $-\sum_{j=q+1}^m \lambda_j(0)>0$, we may choose $b$ sufficiently close to $1$ and $a$ sufficiently close to $0$ so that \eqref{eq:weak_z_q} holds in a neighborhood of $p$.  If the Levi-form has at least $n-q$ positive eigenvalues, then $m\leq q-1$, so
  \[
    \sum_{j=1}^q\lambda_j(0)-\sum_{j,k=1}^n\Upsilon^{\bar k j}(0)\frac{\partial^2\rho}{\partial z_j\partial\bar z_k}(0)=\sum_{j=m+1}^q\lambda_j(0)+\sum_{j=1}^m(1-b)\lambda_j(0)-\sum_{j=m+1}^{n-1}a\lambda_j(0).
  \]
  Since $\lambda_q(0)>0$, we have $\sum_{j=m+1}^q\lambda_j(0)>0$, and hence we may once again choose $b$ sufficiently close to $1$ and $a$ sufficiently close to $0$ so that \eqref{eq:weak_z_q} holds in a neighborhood of $p$.

  We have shown that $\Upsilon$ satisfies the hypotheses of Corollary \ref{cor:sufficient_condition_C_2}, and hence maximal estimates hold for $(0,q)$-forms in a neighborhood of $p$.  The converse follows from Corollary \ref{cor:Z_q}.
\end{proof}

In the language of the present paper, \cite{Derr78} and \cite{Ben00} carry out integration by parts using $\Upsilon$ with an eigenvalue in $(0,1)$ with multiplicity $n-1$.  If we apply this method to a domain which is not necessarily pseudoconvex, we obtain Theorem \ref{thm:almost_pseudoconvex}, as we will see:
\begin{proof}[Proof of Theorem \ref{thm:almost_pseudoconvex}]
  Since \eqref{eq:almost_pseudoconvex} and \eqref{eq:almost_pseudoconcave} are independent of the choice of defining function, we let $\rho$ be the defining function for $\Omega$ given by Lemma \ref{lem:rho_derivative_estimate}.

  We first assume that \eqref{eq:almost_pseudoconvex} holds.  Note that this necessarily implies that $\lambda_q\geq 0$ on $\partial\Omega\cap U$.  Assume that $U$ is sufficiently small so that \eqref{eq:necessary_condition} holds on $U$ for some $A>2$.  Fix $0<t<\frac{\epsilon}{(1+\epsilon)A-2}$.  Observe that this implies $0<t<\frac{1}{A}<\frac{1}{2}$.  For $1\leq j,k\leq n$, set
  \begin{equation}
  \label{eq:Upsilon_identity}
    \Upsilon^{\bar k j}(z)=t\left(\delta_{jk}-|\partial\rho(z)|^{-2}\frac{\partial\rho}{\partial z_k}(z)\frac{\partial\rho}{\partial\bar z_j}(z)\right)
  \end{equation}
  for all $z\in U$.  We immediately obtain \eqref{eq:Upsilon_normal_vanishes}.  As in the proof of Lemma \ref{lem:orthonormal_coordinates}, we may show that $\Upsilon^{\bar k j}\in C^1(U)\cap W^{2,1}(U)$ for all $1\leq j,k\leq n$.  Hence $\Upsilon^j\in L^\infty(U)$ for all $1\leq j\leq n$.  For $1\leq j,k\leq n$, we also have $\Upsilon^{\bar k j}=t^{-1}\sum_{\ell=1}^n \Upsilon^{\bar k\ell}\Upsilon^{\bar\ell j}$, so \eqref{eq:Theta_defined} can be used to show that $\Theta_{\Upsilon,t^{-1}}\in L^\infty(U)$.  Observe that $\Upsilon$ has an eigenvalue of $t$ with multiplicity $n-1$ and an eigenvalue of zero with multiplicity one.  We conclude that $\Upsilon\in\mathcal{M}^2_{\Omega,t^{-1}}(U)$.

  For $z\in\partial\Omega\cap U$, we set
  \begin{align*}
    f(z)&=\sum_{\{1\leq j\leq q:\lambda_j<0\}}(-\lambda_j(z)),\\
    g(z)&=\sum_{\{1\leq j\leq q:\lambda_j>0\}}\lambda_j(z)\text{, and}\\
    h(z)&=\sum_{j=q+1}^{n-1}\lambda_j(z).
  \end{align*}
  Each of these functions is a sum of non-negative terms, so each function is non-negative valued.  With this notation, \eqref{eq:necessary_condition} can be written $g(z)\geq\frac{1}{A}(f(z)+g(z)+h(z))$ and \eqref{eq:almost_pseudoconvex} can be written $-f(z)+g(z)\geq\epsilon f(z)$.  For $z\in\partial\Omega\cap U$, we compute
  \[
    \sum_{j=1}^q\lambda_j(z)-\sum_{j,k=1}^n\Upsilon^{\bar k j}(z)\frac{\partial^2\rho}{\partial z_j\partial\bar z_k}(z)=(1-t)(-f(z)+g(z))-th(z).
  \]
  Since \eqref{eq:necessary_condition} implies $-h(z)\geq f(z)+(1-A)g(z)$, we have
  \[
    \sum_{j=1}^q\lambda_j(z)-\sum_{j,k=1}^n\Upsilon^{\bar k j}(z)\frac{\partial^2\rho}{\partial z_j\partial\bar z_k}(z)\geq-(1-2t)f(z)+(1-tA)g(z).
  \]
  Now \eqref{eq:almost_pseudoconvex} implies $g(z)\geq(1+\epsilon)f(z)$, so since $1-tA>0$ we have
  \[
    \sum_{j=1}^q\lambda_j(z)-\sum_{j,k=1}^n\Upsilon^{\bar k j}(z)\frac{\partial^2\rho}{\partial z_j\partial\bar z_k}(z)\geq(\epsilon-t((1+\epsilon)A-2))f(z).
  \]
  Since $f(z)\geq 0$ and we have chosen $t$ so that $\epsilon-t((1+\epsilon)A-2)>0$, we have \eqref{eq:weak_z_q}.  Thus we have
  satisfied the hypotheses of Proposition \ref{prop:sufficient_condition}.

  Next, we assume that \eqref{eq:almost_pseudoconcave} holds.  Note that this necessarily implies that $\lambda_{q+1}\leq 0$ on $\partial\Omega\cap U$.  Assume that $U$ is sufficiently small so that \eqref{eq:necessary_condition} holds on $U$ for some $A>2$.  Fix $1-\frac{\epsilon}{(1+\epsilon)A-2}<t<1$.  Observe that this implies $1>t>1-\frac{1}{A}>\frac{1}{2}$.  For $1\leq j,k\leq n$ and $z\in U$, we define $\Upsilon^{\bar k j}$ by \eqref{eq:Upsilon_identity}.  As before, we obtain $\Upsilon\in\mathcal{M}^2_{\Omega,t^{-1}}$.  For $z\in\partial\Omega\cap U$, we set
  \begin{align*}
    f(z)&=\sum_{j=1}^q(-\lambda_j(z)),\\
    g(z)&=\sum_{\{q+1\leq j\leq n-1:\lambda_j<0\}}(-\lambda_j(z))\text{, and}\\
    h(z)&=\sum_{\{q+1\leq j\leq n-1:\lambda_j>0\}}\lambda_j(z).
  \end{align*}
  In the context of \eqref{eq:almost_pseudoconcave}, each of these functions is non-negative valued, \eqref{eq:necessary_condition} can be written $g(z)\geq\frac{1}{A}(f(z)+g(z)+h(z))$, and \eqref{eq:almost_pseudoconcave} can be written $g(z)-h(z)\geq\epsilon h(z)$.  For $z\in\partial\Omega\cap U$, we compute
  \[
    \sum_{j=1}^q\lambda_j(z)-\sum_{j,k=1}^n\Upsilon^{\bar k j}(z)\frac{\partial^2\rho}{\partial z_j\partial\bar z_k}(z)=-(1-t)f(z)+tg(z)-th(z).
  \]
  Since \eqref{eq:necessary_condition} implies $-f(z)\geq (1-A)g(z)+h(z)$, we have
  \[
    \sum_{j=1}^q\lambda_j(z)-\sum_{j,k=1}^n\Upsilon^{\bar k j}(z)\frac{\partial^2\rho}{\partial z_j\partial\bar z_k}(z)\geq(1-(1-t)A)g(z)+(1-2t)h(z).
  \]
  Now \eqref{eq:almost_pseudoconcave} implies $g(z)\geq(1+\epsilon)h(z)$, so since $1-(1-t)A>0$ we have
  \[
    \sum_{j=1}^q\lambda_j(z)-\sum_{j,k=1}^n\Upsilon^{\bar k j}(z)\frac{\partial^2\rho}{\partial z_j\partial\bar z_k}(z)\geq(\epsilon-(1-t)((1+\epsilon)A-2))h(z).
  \]
  Since $h(z)\geq 0$ and we have chosen $t$ so that $\epsilon-(1-t)((1+\epsilon)A-2)>0$, we have \eqref{eq:weak_z_q}.  Once again, we have satisfied the hypotheses of Proposition \ref{prop:sufficient_condition}.

\end{proof}
Although we have omitted any discussion of the constant $A$ in the statement of Theorem \ref{thm:almost_pseudoconvex}, it can be seen in the proof that if \eqref{eq:necessary_condition} holds for some $A>2$, then we obtain \eqref{eq:maximal_estimate} with $A$ replaced by any constant $\tilde A>\frac{(1+\epsilon)A-2}{\epsilon}$.  In the pseudoconvex or pseudoconcave cases, we can take $\epsilon$ to be arbitrarily large, which means that we can take $\tilde A$ to be arbitrarily close to $A$.

\section{Examples}

When working with explicit examples, it will be helpful to have techniques for computing the trace and determinant of the Levi-form without the need to explicitly compute an orthonormal basis for $T^{1,0}(\partial\Omega)$.  Let $\Omega\subset\mathbb{C}^n$ be a domain with $C^2$ boundary and let $\rho$ be a $C^2$ defining function for $\Omega$.  Let $(a_j^k(z))_{1\leq j,k\leq n}$ be a unitary matrix parameterized by $z\in\partial\Omega$ such that $\set{\sum_{k=1}^n a_j^k(z)\frac{\partial}{\partial z_k}}_{1\leq j\leq n-1}$ is an orthonormal basis for $T^{1,0}(\partial\Omega)$ and $a_n^k(z)=|\partial\rho(z)|^{-1}\frac{\partial\rho}{\partial\bar z_k}(z)$ for all $1\leq k\leq n$.  Then the trace of the Levi-form with respect to the defining function $\rho$ is given by
\[
  \Tr\mathcal{L}(z)=\sum_{j=1}^{n-1}\sum_{k,\ell=1}^n a_j^k(z)\frac{\partial^2\rho}{\partial z_k\partial\bar z_\ell}(z)\bar a_j^\ell(z).
\]
Since $(a_j^k)_{1\leq j,k\leq n}$ is unitary, $\sum_{\ell=1}^n a_\ell^j\bar a_\ell^k=\delta_{jk}$ for all $1\leq j,k\leq n$, and hence
\begin{equation}
\label{eq:trace_formula}
  \Tr\mathcal{L}=\sum_{j=1}^{n}\frac{\partial^2\rho}{\partial z_j\partial\bar z_j}-|\partial\rho|^{-2}\sum_{k,\ell=1}^n \frac{\partial\rho}{\partial\bar z_k}\frac{\partial^2\rho}{\partial z_k\partial\bar z_\ell}\frac{\partial\rho}{\partial z_\ell}\text{ on }\partial\Omega.
\end{equation}

To compute the determinant of the Levi-form, suppose that $p\in\partial\Omega$ is a point at which $\frac{\partial\rho}{\partial z_n}(p)\neq 0$.  Let $U$ be a neighborhood of $p$ such that $\frac{\partial\rho}{\partial z_n}\neq 0$ on $U$.  For $1\leq j\leq n-1$, we may define $L_j=\frac{\partial}{\partial z_j}-\left(\frac{\partial\rho}{\partial z_n}\right)^{-1}\frac{\partial\rho}{\partial z_j}\frac{\partial}{\partial z_n}$, so that $\{L_j\}_{1\leq j\leq n-1}$ is a basis for $T^{1,0}(\partial\Omega\cap U)$.  This is not orthonormal, since $\left<L_j,L_k\right>=\delta_{jk}+\abs{\frac{\partial\rho}{\partial z_n}}^{-2}\frac{\partial\rho}{\partial z_j}\frac{\partial\rho}{\partial\bar z_k}$ for all $1\leq j,k,\leq n-1$, but we may compute the determinant of the $(n-1)\times(n-1)$ matrix
\[
  \det\left(\left<L_j,L_k\right>\right)_{1\leq j,k\leq n-1}=1+\abs{\frac{\partial\rho}{\partial z_n}}^{-2}\sum_{j=1}^{n-1}\abs{\frac{\partial\rho}{\partial z_j}}^2=\abs{\frac{\partial\rho}{\partial z_n}}^{-2}|\partial\rho|^2.
\]
Suppose that $(b_j^k(z))_{1\leq j,k\leq n-1}$ is a matrix parameterized by $z\in\partial\Omega\cap U$ such that $\set{\sum_{k=1}^n b_j^k(z)L_k}_{1\leq j\leq n-1}$ is an orthonormal basis for $T^{1,0}(\partial\Omega\cap U)$.  Since the determinant is multiplicative, we must have $\abs{\det(b_j^k)_{1\leq j,k\leq n-1}}=\abs{\frac{\partial\rho}{\partial z_n}}|\partial\rho|^{-1}$.  If we compute the determinant of the Levi-form in these orthonormal coordinates, we may use \eqref{eq:hessian_alternate} to write
\[
  \det\mathcal{L}=\det\left(\sum_{\ell,m=1}^{n-1}b_j^\ell\left(\Hess(L_\ell,\bar L_m)\rho\right)\bar b_k^m\right)_{1\leq j,k\leq n-1},
\]
so since the determinant is multiplicative, we have
\begin{equation}
\label{eq:determinant_formula}
  \det\mathcal{L}=\abs{\frac{\partial\rho}{\partial z_n}}^2|\partial\rho|^{-2}\det\left(\Hess(L_j,\bar L_k)\rho\right)_{1\leq j,k\leq n-1}\text{ on }\partial\Omega\cap U.
\end{equation}

Our first example demonstrates that the values of $A$ in our necessary condition and our sufficient condition are sharp:
\begin{prop}
  For $t\geq 1$, let $\Omega\subset\mathbb{C}^3$ be the domain defined by the function
  \[
    \rho_t(z)=\frac{2}{3}(\re z_1)^3-t|z_2|^2\re z_1-\im z_3.
  \]
  Then \eqref{eq:necessary_condition} holds for $q=2$ on some neighborhood of the origin only if $A>1+t$, and for every $A>1+t$ there exists a neighborhood $U$ of the origin on which \eqref{eq:maximal_estimate} holds for $q=2$ and some $B>0$.
\end{prop}

\begin{proof}
We compute
\[
  \dbar\rho_t(z)=\left((\re z_1)^2-\frac{t}{2}|z_2|^2\right)\, d\bar z_1-t z_2\re z_1 \, d\bar z_2-\frac{i}{2} \, d\bar z_3,
\]
and
\[
  \ddbar\rho_t(z)=\re z_1 \, dz_1\wedge d\bar z_1-\frac{t}{2} \bar z_2 \, dz_2\wedge d\bar z_1-\frac{t}{2}z_2\, dz_1\wedge d\bar z_2-t\re z_1 \, dz_2\wedge d\bar z_2.
\]
Using \eqref{eq:trace_formula}, we compute
\begin{multline*}
  \Tr\mathcal{L}=(1-t)\re z_1-|\partial\rho_t|^{-2}\left(\abs{(\re z_1)^2-\frac{t}{2}|z_2|^2}^2\re z_1\right)\\
  -|\partial\rho_t|^{-2}\left(t^2(\re z_1)|z_2|^2\left((\re z_1)^2-\frac{t}{2}|z_2|^2\right)-t^3(\re z_1)^3|z_2|^2\right).
\end{multline*}
In the notation of \eqref{eq:determinant_formula}, \eqref{eq:hessian_alternate} implies $\Hess(L_j,\bar L_k)\rho_t=\frac{\partial^2\rho_t}{\partial z_j\partial\bar z_k}$ for all $1\leq j,k\leq 2$, so \eqref{eq:determinant_formula} implies
\[
  \det\mathcal{L}=-\frac{1}{4}|\partial\rho_t|^{-2}\left(t(\re z_1)^2+\frac{t^2}{4}|z_2|^2\right).
\]
Observe that $\det\mathcal{L}\leq 0$, so we must always have at least one non-positive and one non-negative eigenvalues.  This means that to show \eqref{eq:necessary_condition} for $q=2$, we must find $A$ satisfying $\lambda_2\geq\frac{1}{A}(\lambda_2-\lambda_1)$ near the origin.

When $z_2=0$, $|\partial\rho_t|^2=(\re z_1)^4+\frac{1}{4}$, so we have
\[
  \Tr\mathcal{L}=(1-t)\re z_1-4\left(4(\re z_1)^4+1\right)^{-1}(\re z_1)^5.
\]
and
\[
  \det\mathcal{L}=-\left(4(\re z_1)^4+1\right)^{-1}t(\re z_1)^2.
\]
Since the Levi-form is represented by a $2\times 2$ matrix, this suffices to characterize the eigenvalues.  In particular, when $z_2=0$ and $\re z_1>0$, we have
\[
  \lambda_1=-t\re z_1\text{ and }\lambda_2=\left(4(\re z_1)^4+1\right)^{-1}\re z_1,
\]
so $A$ must satisfy
\[
  A\geq 1-\frac{\lambda_1}{\lambda_2}=1+t\left(4(\re z_1)^4+1\right).
\]
Hence, \eqref{eq:necessary_condition} holds on a sufficiently small neighborhood of the origin only if $A>t+1$.

For the converse, we fix $A>t+1$.  Define $\Upsilon^{\bar k j}$ for $1\leq j,k\leq 2$ by $\Upsilon^{\bar 1 1}=\frac{1}{1+t}-\frac{1}{2}\re z_1$, $\Upsilon^{\bar 2 2}=\frac{t}{1+t}+\frac{1}{2t}\re z_1$, $\Upsilon^{\bar 2 1}=\frac{1}{t}\bar z_2$, and $\Upsilon^{\bar 1 2}=\frac{1}{t}z_2$.  For $1\leq j\leq 2$, we define $\Upsilon^{\bar 3 j}=-2i\sum_{k=1}^2\frac{\partial\rho_t}{\partial\bar z_k}\Upsilon^{\bar k j}$ and $\Upsilon^{\bar j 3}=\overline{\Upsilon^{\bar 3 j}}$.  We define $\Upsilon^{\bar 3 3}=2i\sum_{j=1}^2\Upsilon^{\bar 3 j}\frac{\partial\rho_t}{\partial z_j}=4\sum_{j,k=1}^2\frac{\partial\rho_t}{\partial\bar z_k}\Upsilon^{\bar k j}\frac{\partial\rho_t}{\partial z_j}$.  We have $\sum_{k=1}^3\frac{\partial\rho_t}{\partial\bar z_k}\Upsilon^{\bar k j}\equiv 0$ for all $1\leq j\leq 3$, so $\Upsilon$ must satisfy \eqref{eq:Upsilon_tangential}.  When $\re z_1=0$ and $z_2=0$, $\Upsilon$ is diagonal with diagonal entries $\frac{1}{1+t}$, $\frac{t}{1+t}$, and $0$.  Hence, for any $0<\epsilon<\frac{1}{1+t}-\frac{1}{A}$, there exists a neighborhood $U$ of the origin on which the non-trivial eigenvalues of $\Upsilon$ lie in the interval
\[
  \left[\frac{1}{1+t}-\epsilon,\frac{t}{1+t}+\epsilon\right]\subset\left(\frac{1}{A},1-\frac{1}{A}\right).
\]
We have
\[
  \sum_{j,k=1}^{3}\Upsilon^{\bar k j}(\rho_t)_{j\bar k}=(1-t)\re z_1-(\re z_1)^2-|z_2|^2,
\]
and
\[
  \abs{\Tr\mathcal{L}-(1-t)\re z_1}\leq O(((\re z_1)^2+|z_2|^2)^{5/2}),
\]
so for any $\epsilon>0$ there exists a neighborhood $U$ of the origin on which
\[
  \Tr\mathcal{L}-\sum_{j,k=1}^{3}\Upsilon^{\bar k j}(\rho_t)_{j\bar k}\geq(1-\epsilon)((\re z_1)^2+|z_2|^2),
\]
from which \eqref{eq:weak_z_q} follows for $q=2$.  We have shown that $\Upsilon$ satisfies the hypotheses of Corollary \ref{cor:sufficient_condition_C_2}, so we have maximal estimates on $(0,2)$-forms on a neighborhood of the origin.
\end{proof}

We conclude with an example demonstrating the gap between Proposition \ref{prop:sufficient_condition} and Corollary \ref{cor:sufficient_condition_C_2}:
\begin{prop}
\label{prop:ex_2}
  For $t\in\mathbb{R}$, let $\Omega_t\subset\mathbb{C}^3$ be the domain defined by the function
  \[
    \rho_t(z)=\frac{1}{4}(|z_1|^4+|z_2|^4)-t|z_1|^2|z_2|^2-\im z_3.
  \]
  Then the hypotheses of Corollary \ref{cor:sufficient_condition_C_2} fail for $q=2$ whenever $t>1$, but we have maximal estimates for $(0,2)$-forms near the origin whenever $t<\frac{13+2\sqrt{31}}{9}$ and the necessary condition \eqref{eq:necessary_condition} holds for $q=2$ whenever $t\in\mathbb{R}$.
\end{prop}

\begin{proof} We start with some preliminary calculations.  We compute
\[
  \dbar\rho_t(z)=\left(\frac{1}{2}|z_1|^2-t |z_2|^2\right)z_1 \, d\bar z_1+\left(\frac{1}{2}|z_2|^2-t |z_1|^2\right)z_2 \, d\bar z_2-\frac{i}{2} \, d\bar z_3
\]
and
\begin{multline*}
  \ddbar\rho_t(z)=\\
  (|z_1|^2-t|z_2|^2) \, dz_1\wedge d\bar z_1-tz_1 \bar z_2 \, dz_2\wedge d\bar z_1-t\bar z_1 z_2 \, dz_1\wedge d\bar z_2+(|z_2|^2-t|z_1|^2)\, dz_2\wedge d\bar z_2.
\end{multline*}
By \eqref{eq:trace_formula},
\begin{multline*}
  \Tr\mathcal{L}=(1-t)(|z_1|^2+|z_2|^2)-|\partial\rho_t|^{-2}\left(\frac{1}{2}|z_1|^2-t |z_2|^2\right)^2|z_1|^2(|z_1|^2-t|z_2|^2)\\
  +|\partial\rho_t|^{-2}2t|z_1|^2|z_2|^2\left(\frac{1}{2}|z_2|^2-t |z_1|^2\right)\left(\frac{1}{2}|z_1|^2-t|z_2|^2\right)\\
  -|\partial\rho_t|^{-2}\left(\frac{1}{2}|z_2|^2-t |z_1|^2\right)^2|z_2|^2(|z_2|^2-t|z_1|^2),
\end{multline*}
and by \eqref{eq:determinant_formula},
\[
  \det\mathcal{L}=\frac{1}{4}|\partial\rho_t|^{-2}\left(-t|z_1|^4+|z_1|^2|z_2|^2-t|z_2|^4\right)
\]
In particular,
\begin{equation}
\label{eq:ex2_trace}
  \abs{\Tr\mathcal{L}-(1-t)(|z_1|^2+|z_2|^2)}\leq O((|z_1|^2+|z_2|^2)^4)
\end{equation}
and
\begin{equation}
\label{eq:ex2_det}
  \abs{\det\mathcal{L}+t(|z_1|^4+|z_2|^4)-|z_1|^2|z_2|^2}\leq O((|z_1|^2+|z_2|^2)^5).
\end{equation}

Fix $t>1$. We will now show that the hypotheses of Corollary \ref{cor:sufficient_condition_C_2} fail with this choice of $t$.  Suppose that $\Upsilon$ is a positive semi-definite Hermitian $3\times 3$ matrix with continuous coefficients that satisfies \eqref{eq:weak_z_q} in a neighborhood of the origin, and has eigenvalues in some relatively compact subset of $(0,1)$ at the origin.  When $z_1=0$, we have
\begin{multline*}
  0\leq\Tr\mathcal{L}-\sum_{j,k=1}^n\Upsilon^{\bar k j}(\rho_t)_{j\bar k}=\\
   (1-t)|z_2|^2
  -\frac{1}{4}|\partial\rho_t|^{-2}|z_2|^8
  +t|z_2|^2\Upsilon^{\bar 1 1}(z)-|z_2|^2\Upsilon^{\bar 2 2}(z).
\end{multline*}
If we divide this by $|z_2|^2$ and let $z$ approach the origin, we have
\begin{equation}
\label{eq:ex2_Upsilon_inequality_1}
  0\leq 1-t+t\Upsilon^{\bar 1 1}(0)-\Upsilon^{\bar 2 2}(0).
\end{equation}
When $z_2=0$, we have
\begin{multline*}
  0\leq\Tr\mathcal{L}-\sum_{j,k=1}^n\Upsilon^{\bar k j}(\rho_t)_{j\bar k}=\\
   (1-t)|z_1|^2-\frac{1}{4}|\partial\rho_t|^{-2}|z_1|^8
  -|z_1|^2\Upsilon^{\bar 1 1}(z)+t|z_1|^2\Upsilon^{\bar 2 2}(z).
\end{multline*}
This time, we divide by $|z_1|^2$ before letting $z$ approach the origin to obtain
\begin{equation}
\label{eq:ex2_Upsilon_inequality_2}
  0\leq 1-t-\Upsilon^{\bar 1 1}(0)+t\Upsilon^{\bar 2 2}(0).
\end{equation}
If we multiple \eqref{eq:ex2_Upsilon_inequality_1} by $t$ and add it to \eqref{eq:ex2_Upsilon_inequality_2}, we obtain
\[
  (1-t^2)\Upsilon^{\bar 1 1}(0)\leq 1-t^2,
\]
while multiplying \eqref{eq:ex2_Upsilon_inequality_2} by $t$ and adding it to \eqref{eq:ex2_Upsilon_inequality_1} gives us
\[
  (1-t^2)\Upsilon^{\bar 2 2}(0)\leq 1-t^2.
\]
Since $1-t^2<0$ by assumption, $\Upsilon^{\bar 1 1}(0)\geq 1$ and $\Upsilon^{\bar 2 2}(0)\geq 1$.  This guarantees that $\Upsilon$ must have at least one eigenvalue greater than or equal to one, contradicting our hypotheses.  Hence, the hypotheses of Corollary \ref{cor:sufficient_condition_C_2} must fail whenever $t>1$.

We now turn to the most difficult calculation in this proof: that the hypotheses of Proposition \ref{prop:sufficient_condition} hold when $t < \frac{13+2\sqrt{31}}9$.  Before constructing $\Upsilon$, we first define the rational functions $f_j(z)=\frac{|z_j|^2}{|z_1|^2+|z_2|^2}$ for $1\leq j\leq 2$ and $g(z)=4f_1(z)f_2(z)$.  Observe that $0\leq f_j(z)\leq 1$ for every $1\leq j\leq 2$ and almost every $z\in\mathbb{C}^3$.  Since
\[
  1-g(z)=\frac{(|z_1|^2-|z_2|^2)^2}{(|z_1|^2+|z_2|^2)^2},
\]
we must also have $0\leq g(z)\leq 1$ for almost every $z\in\mathbb{C}^3$.  Hence, $f_1,f_2,g\in L^\infty(\mathbb{C}^3)$, even though each function is discontinuous at any point for which $z_1=z_2=0$.  When $|z_1|^2+|z_2|^2\neq 0$, we have
\begin{equation}
  \label{eq:f_sum_identity}f_1(z)+f_2(z)=1.
\end{equation}
At such points, we have
\[
  (f_1(z))^2+(f_2(z))^2=(f_1(z)+f_2(z))^2-\frac{1}{2}g(z),
\]
so \eqref{eq:f_sum_identity} gives us
\begin{equation}
  \label{eq:f_square_sum_identity}(f_1(z))^2+(f_2(z))^2=1-\frac{1}{2}g(z).
\end{equation}
Similarly,
\[
  (f_1(z))^3+(f_2(z))^3=(f_1(z)+f_2(z))^3-\frac{3}{4}g(z)(f_1(z)+f_2(z)),
\]
so \eqref{eq:f_sum_identity} gives us
\begin{equation}
  \label{eq:f_cube_sum_identity}(f_1(z))^3+(f_2(z))^3=1-\frac{3}{4}g(z).
\end{equation}

Fix $a$ and $b$ such that $0<a<1$ and $0<b<\min\set{\frac{1}{4}a,1-a}$.  When $|z_1|^2+|z_2|^2\neq 0$ we set
\[
  \psi(z)=(|z_1|^2+|z_2|^2)(a-bg(z)),
\]
and otherwise we set $\psi(z)=0$.  We have
\begin{equation}
\label{eq:psi_bounds}
  (a-b)(|z_1|^2+|z_2|^2)\leq\psi(z)\leq a(|z_1|^2+|z_2|^2)\text{ for all }z\in \mathbb{C}^3.
\end{equation}
Observe that if $D^k$ is a $k$th-order differential operator for some integer $k\geq 0$, then we have $|D^k\psi(z)|\leq O((\sqrt{|z_1|^2+|z_2|^2})^{2-k})$ on $U$.  Hence, standard integral estimates imply that $\psi\in C^{1,1}(U)\cap W^{3,4-\epsilon}(U)\cap W^{4,2-\epsilon}(U)$ for any $\epsilon>0$, so $\psi\in C^{1,1}(U)\cap W^{3,2}(U)\cap W^{4,1}(U)$ (note that since $\psi$ is independent of $z_3$, we can reduce these to estimates on $\mathbb{R}^4$, and hence the critical exponent for integrability is four).  Since
\[
  (|z_1|^2+|z_2|^2)g(z)=4|z_2|^2-\frac{4|z_2|^4}{|z_1|^2+|z_2|^2},
\]
we may easily compute
\[
  \frac{\partial\psi}{\partial z_1}(z)=\left(a-4b(f_2(z))^2\right)\bar z_1.
\]
Similarly, we obtain
\[
  \frac{\partial\psi}{\partial z_2}(z)=\left(a-4b(f_1(z))^2\right)\bar z_2.
\]
For $1\leq j,k\leq 2$ satisfying $j\neq k$, we compute $\bar z_j\frac{\partial}{\partial\bar z_j}f_k(z)=-\frac{1}{4}g(z)$.  With this in mind, we may compute
\begin{align*}
  \frac{\partial^2\psi}{\partial z_1\partial\bar z_1}(z)&=a-4b(f_2(z))^2+2bf_2(z)g(z)\text{ and}\\
  \frac{\partial^2\psi}{\partial z_2\partial\bar z_2}(z)&=a-4b(f_1(z))^2+2bf_1(z)g(z).
\end{align*}
For $1\leq j,k\leq 2$ satisfying $j\neq k$, we also have $\frac{\partial}{\partial\bar z_j}f_j(z)=\frac{z_j}{|z_1|^2+|z_2|^2}f_k(z)$, so
\begin{align*}
  \frac{\partial^2\psi}{\partial z_1\partial\bar z_2}(z)&=\frac{-2b z_2\bar z_1 g(z)}{|z_1|^2+|z_2|^2}\text{ and}\\
  \frac{\partial^2\psi}{\partial z_2\partial\bar z_1}(z)&=\frac{-2b z_1\bar z_2 g(z)}{|z_1|^2+|z_2|^2}.
\end{align*}

To build $\Upsilon$, our basic building block will be a $2\times 2$ Hermitian matrix $M$ defined almost everywhere on $U$ by
\[
M = \big(M^{\bar kj}\Big)_{1\leq j,k\leq 2} =  \left(\Tr\Big(\frac{\p^2\psi}{\p z_k\p\z_j}\Big)_{1\leq j,k\leq 2}\right)I - \Big(\frac{\p^2\psi}{\p z_k\p\z_j}\Big)_{1\leq j,k\leq 2},
\]
i.e.,
\[
M = \begin{pmatrix} \frac{\p^2\psi}{\p z_2\p\z_2} & -\frac{\p^2\psi}{\p z_2\p\z_1} \\ -\frac{\p^2\psi}{\p z_1\p\z_2} & \frac{\p^2\psi}{\p z_1\p\z_1}\end{pmatrix}.
\]
On $\partial\Omega$, $M$ will be approximately equal to the restriction of $\Upsilon$ to $T^{1,0}(\partial\Omega)\otimes T^{0,1}(\partial\Omega)$.  For each $1\leq j,k\leq 2$, regularity properties of $\psi$ imply that $M^{\bar k j}\in L^\infty(U)\cap W^{1,2}(U)\cap W^{2,1}(U)$.  Using symmetry properties of weak derivatives, we have
\begin{equation}
\label{eq:M_divergence}
  \sum_{k=1}^2\frac{\partial}{\partial\bar z_k}M^{\bar k j}\equiv 0\text{ for each }1\leq j\leq 2.
\end{equation}

We now show that both eigenvalues of $M$ lie in $(1/A,1-1/A)$ almost everywhere.  Using \eqref{eq:f_sum_identity} and \eqref{eq:f_square_sum_identity} to simplify, we obtain
\[
  \det M(z)=a^2-4ab(1-g(z))-b^2(g(z))^2
\]
and
\[
  \Tr M(z)=2a-4b(1-g(z)).
\]
Since $M$ is a $2\times 2$ matrix, we know that if
\[
  \lambda^2-\lambda\Tr M(z)+\det M(z)=\det(\lambda I-M(z))\geq 0
\]
almost everywhere on $\mathbb{C}^3$, then both eigenvalues of $M(z)$ are either uniformly bounded above or below by $\lambda$.  We may evaluate the sign of
\[
  2\lambda-\Tr M(z)=\Tr(\lambda I-M(z))
\]
to determine whether $\lambda$ is an upper bound or a lower bound.  Since $0\leq g(z)\leq 1$, we may compute
\[
  (a-4b)^2-(a-4b)\Tr M(z)+\det M(z)=16b^2g(z)-b^2(g(z))^2\geq 0
\]
and
\[
  \Tr M(z)-2(a-4b)=4bg(z)+4b\geq 0,
\]
to see that each eigenvalue of $M(z)$ is bounded below by $a-4b>0$ almost everywhere.  On the other hand,
\begin{multline*}
  (a+b)^2-(a+b)\Tr M(z)+\det M(z)=5b^2-4b^2g(z)-b^2(g(z))^2\\
  =b^2(5+g(z))(1-g(z))\geq 0
\end{multline*}
and
\[
  2(a+b)-\Tr M(z)=6b-4bg(z)\geq 0,
\]
so each eigenvalue of $M(z)$ is bounded above by $a+b<1$ almost everywhere.  Hence, for any $A>\max\set{\frac{1}{a-4b},\frac{1}{1-a-b}}$, both eigenvalues of $M(z)$ lie in $(1/A,1-1/A)$ almost everywhere.  Since these eigenvalues are not continuous, it will be helpful to note that they are uniformly bounded away from $\frac{1}{A}$ and $1-\frac{1}{A}$.

We are now ready to define $\Upsilon$.  For $1\leq j,k\leq 2$, we define $\Upsilon^{\bar k j}=M^{\bar k j}$.  For $1\leq j\leq 2$, we define $\Upsilon^{\bar 3 j}=-2i\sum_{k=1}^2\frac{\partial\rho_t}{\partial\bar z_k}M^{\bar k j}$ and $\Upsilon^{\bar j 3}=\overline{\Upsilon^{\bar 3 j}}$.  Finally, we define $\Upsilon^{\bar 3 3}=4\sum_{j,k=1}^2\frac{\partial\rho_t}{\partial\bar z_k}M^{\bar k j}\frac{\partial\rho_t}{\partial z_j}$.  Observe that
\begin{equation}
\label{eq:Upsilon_kernel}
  \sum_{k=1}^3\frac{\partial\rho_t}{\partial\bar z_k}\Upsilon^{\bar k j}\equiv 0\text{ for all }1\leq j\leq 3,
\end{equation}
so we immediately obtain \eqref{eq:Upsilon_normal_vanishes}.

To estimate the eigenvalues of $\Upsilon$, we first observe that for $1\leq k\leq 2$, $\frac{\partial\rho_t}{\partial\bar z_k}$ is a continuous function vanishing at the points of discontinuity for $M$, so for $1\leq j\leq 3$, $\Upsilon^{\bar 3 j}$ is a continuous function vanishing whenever $z_1=z_2=\im z_3=0$.  Hence, on a sufficiently small neighborhood of the origin, $\Upsilon$ is an arbitrarily small perturbation of the matrix
\[
  \Upsilon_0=\begin{pmatrix}M^{\bar 1 1}&M^{\bar 1 2}&0\\M^{\bar 2 1}&M^{\bar 2 2}&0\\0&0&0\end{pmatrix}.
\]
Since $\Upsilon_0$ has two eigenvalues in the interval $(1/A,1-1/A)$ (uniformly bounded away from the endpoints), we can choose $U$ sufficiently small so that $\Upsilon$ has two eigenvalues in $(1/A,1-1/A)$ (see Corollary 6.3.4 in \cite{HoJo85} for a result on perturbations of eigenvalues that suffices for our purposes).  By \eqref{eq:Upsilon_kernel}, zero is also an eigenvalue of $\Upsilon$. Since $\Upsilon$ has only three eigenvalues, we conclude that $\Upsilon$ has two eigenvalues in $(1/A,1-1/A)$ and one eigenvalue equal to zero, and hence $\Upsilon$ and $I-\Upsilon$ are positive semi-definite on $U$.

It remains for us to show that $\Upsilon$ has the regularity required to be an element of $\mathcal{M}^2_{\Omega,\eta}(U)$ for all $0 \leq \eta \leq 1$.  For $1\leq j\leq 2$, \eqref{eq:M_divergence} implies $\Upsilon^{j}\equiv 0$ almost everywhere.  Here, we have used the fact that every coefficient of $\dbar\rho_t$ and $M$ is independent of $z_3$.  For the third entry, we again use \eqref{eq:M_divergence} to compute
\[
  \Upsilon^{3}=\sum_{j,k=1}^2 2i M^{\bar k j}\frac{\partial^2\rho_t}{\partial z_j\partial\bar z_k}.
\]
We clearly have $\Upsilon^j\in L^\infty(U)$ for all $1\leq j\leq 3$.  For $z\in\mathbb{R}^n$, we set
\[
  \mu(z)=\begin{cases}\sum_{j,k=1}^2\frac{\partial^2\rho_t}{\partial z_j\partial\bar z_k}(z)M^{\bar k j}(z)&|z_1|^2+|z_2|^2\neq 0\\0&|z_1|^2+|z_2|^2=0\end{cases},
\]
so that $\Upsilon^{3}=2i\mu$ almost everywhere on $U$.  If we can show that $\mu\in C^{0,1}(U)$, then \eqref{eq:Theta_identity} implies $\Theta_{\Upsilon,\eta}\in L^\infty(U)$ for any $0\leq\eta\leq 1$.

To compute $\mu(z)$, we first use \eqref{eq:f_sum_identity}, \eqref{eq:f_square_sum_identity}, and \eqref{eq:f_cube_sum_identity} to obtain
\[
  |z_1|^2\frac{\partial^2\psi}{\partial z_2\partial\bar z_2}(z)+|z_2|^2\frac{\partial^2\psi}{\partial z_1\partial\bar z_1}(z)=(|z_1|^2+|z_2|^2)\left(a-4b+5bg(z)-b(g(z))^2\right)
\]
and
\[
  |z_2|^2\frac{\partial^2\psi}{\partial z_2\partial\bar z_2}(z)+|z_1|^2\frac{\partial^2\psi}{\partial z_1\partial\bar z_1}(z)=(|z_1|^2+|z_2|^2)\left(a-bg(z)+b(g(z))^2\right).
\]
Since
\[
  z_1\bar z_2\frac{\partial^2\psi}{\partial z_1\partial\bar z_2}(z)=-\frac{b}{2}(|z_1|^2+|z_2|^2)(g(z))^2,
\]
we have the tools to compute
\begin{equation}
\label{eq:mu_computation}
  \mu(z)=(|z_1|^2+|z_2|^2)\left((1-t)a-4b+(5+t)bg(z)-(1+2t)b(g(z))^2\right)
\end{equation}
when $|z_1|^2+|z_2|^2\neq 0$.  We have $\abs{\mu(z)}\leq O(|z_1|^2+|z_2|^2)$, so it is not difficult to check that $\mu\in C^{1,1}(U)$
which means that $\Upsilon\in\mathcal{M}^2_{\Omega,\eta}(U)$ for all $0\leq\eta\leq 1$.

To complete our proof that $\Upsilon$ satisfies the hypotheses of Proposition \ref{prop:sufficient_condition}, we must show that \eqref{eq:weak_z_q} holds.  By construction we have $\sum_{j,k=1}^3\Upsilon^{\bar k j}\frac{\partial^2\rho_t}{\partial z_j\partial\bar z_k}=\mu$, so it will suffice to prove
\begin{equation}
\label{eq:weak_z_q_mu}
  \Tr\mathcal{L}-\mu\geq 0\text{ on }U.
\end{equation}
For $t>-\frac{1}{2}$, we note that the vertex of the parabola $y=(5+t)x-(1+2t)x^2$ is located at $\left(\frac{5+t}{2(1+2t)},\frac{(5+t)^2}{4(1+2t)}\right)$, and $0<\frac{5+t}{2(1+2t)}<1$ if and only if $t>1$.  For $t\leq-\frac{1}{2}$, this parabola has no local maxima.  Hence, for $0\leq x\leq 1$, we have $(5+t)x-(1+2t)x^2\leq\frac{(5+t)^2}{4(1+2t)}$ when $t>1$ and $(5+t)x-(1+2t)x^2\leq 4-t$ when $t\leq 1$.  Replacing $x$ with $g(z)$, we see that \eqref{eq:mu_computation} implies
\[
  \mu(z)\leq\left((1-t)a-4b+\frac{(5+t)^2}{4(1+2t)}b\right)(|z_1|^2+|z_2|^2)
\]
when $t>1$ and
\[
  \mu(z)\leq\left((1-t)a-tb\right)(|z_1|^2+|z_2|^2)
\]
when $t\leq 1$.

When $t<1$, we may set $a=\frac{1}{2}$ and $b=0$ (which means that $\Upsilon$ has smooth coefficients), and \eqref{eq:ex2_trace} implies that
\[
  \Tr\mathcal{L}(z)-\mu(z)\geq\frac{1-t}{2}(|z_1|^2+|z_2|^2)-O((|z_1|^2+|z_2|^2)^4),
\]
so we may choose $U$ sufficiently small so that \eqref{eq:weak_z_q_mu} (and hence \eqref{eq:weak_z_q}) holds.  This means that the hypotheses of Corollary \ref{cor:sufficient_condition_C_2} are satisfied for any $A>2$, so we have maximal estimates on $(0,2)$-forms.

When $1\leq t<\frac{13+2\sqrt{31}}{9}$, we have $9t^2-26t+5<0$.  Since $t>\frac{1}{2}$, elementary algebraic manipulations may be used to show that this is equivalent to $1-t>-4+\frac{(5+t)^2}{4(1+2t)}$.  Since $t\geq 1$, $-4+\frac{(5+t)^2}{4(1+2t)}<0$, and so we have $(1-t)\left(-4+\frac{(5+t)^2}{4(1+2t)}\right)^{-1}<1$.  If we set $a=\frac{4}{5}$, then $\Upsilon$ satisfies our eigenvalue hypothesis whenever $0\leq b<\frac{1}{5}$.  Hence, we may choose $b$ satisfying $\frac{1-t}{5}\left(-4+\frac{(5+t)^2}{4(1+2t)}\right)^{-1}<b<\frac{1}{5}$.  By \eqref{eq:ex2_trace}, we have
\begin{multline*}
  \Tr\mathcal{L}(z)-\mu(z)\geq\\
  \left(\frac{1-t}{5}+\left(4-\frac{(5+t)^2}{4(1+2t)}\right)b\right)(|z_1|^2+|z_2|^2)-O((|z_1|^2+|z_2|^2)^4).
\end{multline*}
Once again, this allows us to choose $U$ sufficiently small so that  \eqref{eq:weak_z_q_mu} holds, and hence the hypotheses of Proposition \ref{prop:sufficient_condition} are satisfied.  As a result, we have maximal estimates on $(0,2)$-forms whenever $t<\frac{13+2\sqrt{31}}{9}$.

We conclude with a proof that \eqref{eq:necessary_condition} holds for all $t\in\mathbb{R}$.  By Theorem \ref{thm:necessary_condition}, we immediately obtain \eqref{eq:necessary_condition} whenever $t<\frac{13+2\sqrt{31}}{9}$.  Hence, it will suffice to prove \eqref{eq:necessary_condition} when $t>1$.  In this case, we have
\[
  t(|z_1|^4+|z_2|^4)-|z_1|^2|z_2|^2\geq\frac{2t-1}{4}(|z_1|^2+|z_2|^2)^2,
\]
so \eqref{eq:ex2_det} implies
\[
  \det\mathcal{L}\leq -\frac{2t-1}{4}(|z_1|^2+|z_2|^2)^2+O((|z_1|^2+|z_2|^2)^{5}).
\]
Since \eqref{eq:ex2_trace} implies that
\[
  (\Tr\mathcal{L})^2\leq(1-t)^2(|z_1|^2+|z_2|^2)^2+O((|z_1|^2+|z_2|^2)^{5}),
\]
we may choose $0<s<\frac{2t-1}{4(1-t)^2}$ and find a neighborhood $U$ of $0$ sufficiently small so that
\[
  \det\mathcal{L}\leq -s(\Tr\mathcal{L})^2
\]
on $U$.  Since this means that $\det\mathcal{L}\leq 0$ on $U$, the Levi-form can have at most one negative and one positive eigenvalue, so \eqref{eq:necessary_condition} is equivalent to $\lambda_2\geq\frac{1}{A}(-\lambda_1+\lambda_2)$ on $U$.  Since the Levi-form is represented by a $2\times 2$ matrix, the eigenvalues are given by $\lambda_1=\frac{1}{2}\Tr\mathcal{L}-\frac{1}{2}\sqrt{(\Tr\mathcal{L})^2-4\det\mathcal{L}}$ and $\lambda_2=\frac{1}{2}\Tr\mathcal{L}+\frac{1}{2}\sqrt{(\Tr\mathcal{L})^2-4\det\mathcal{L}}$, so \eqref{eq:necessary_condition} is equivalent to
\[
  \Tr\mathcal{L}\geq-\left(1-\frac{2}{A}\right)\sqrt{(\Tr\mathcal{L})^2-4\det\mathcal{L}}.
\]
If $\Tr\mathcal{L}\geq 0$, then this holds for all $A>2$.  If $\Tr\mathcal{L}<0$, then we have shown that
\[
  \sqrt{(\Tr\mathcal{L})^2-4\det\mathcal{L}}\geq-\sqrt{1+4s}\Tr\mathcal{L}
\]
on $U$, so it suffices to prove that $1\leq\left(1-\frac{2}{A}\right)\sqrt{1+4s}$, and this holds whenever $A\geq \frac{2\sqrt{1+4s}}{\sqrt{1+4s}-1}$.  Hence, \eqref{eq:necessary_condition} holds.

\end{proof}
\bibliographystyle{amsplain}
\bibliography{mybib3-12-19}
\end{document}